\documentclass[11pt]{amsart}

\usepackage{amsmath,amscd,amssymb}
\usepackage[mathscr]{eucal}
\usepackage[all]{xy}

\newcommand{\Wip}{\mathrm{A}_+^1}

\newcommand{\vanish}[1]{\relax}

\newcommand{\N}{\mathbb{N}}

\newcommand{\R}{\mathbb{R}}
\newcommand{\C}{\mathbb{C}}

\newcommand{\ud}{\mathrm{d}}

\newcommand{\Sum}[2][\relax]{%
 \ifx#1\relax \sideset{}{_{#2}}\sum 
 \else \sideset{}{^{#1}_{#2}}\sum
 \fi}

\DeclareMathOperator{\re}{Re}

\newcommand{\Ce}{\mathrm{C}}

\newcommand{\tpfeil}{\longmapsto}
\DeclareMathOperator{\dom}{dom}
\DeclareMathOperator{\ran}{ran}

\newcommand{\cls}[1]{\overline{#1}}

\newcommand{\norm}[2][\relax]{%
   \ifx#1\relax \ensuremath{\left\Vert#2\right\Vert}
   \else \ensuremath{\left\Vert#2\right\Vert_{#1}}
   \fi}

\makeatletter
\newcommand{\sprod}[2]{\ensuremath{%
  \setbox0=\hbox{\ensuremath{#2}}
  \dimen@\ht0
  \advance\dimen@ by \dp0
  \left(\left.#1\rule[-\dp0]{0pt}{\dimen@}\,\right|#2\hspace{1pt}\right)}}
 \makeatother

\newcounter{aufzi}

\newcounter{aufzii}

\newcounter{aufziii}

 \newtheorem{thm}{Theorem}[section]
 \newtheorem{cor}[thm]{Corollary}
 \newtheorem{lemma}[thm]{Lemma}
 \newtheorem{prop}[thm]{Proposition}

 \theoremstyle{definition}
 \newtheorem{defn}[thm]{Definition}

 \theoremstyle{remark}
 \newtheorem{rem}[thm]{Remark}

\newtheorem{exa}[thm]{Example}
\newtheorem{remark}[thm]{Remark}

\numberwithin{equation}{section}

\newtheorem{theorem}{Theorem}[section]

\numberwithin{equation}{section} \numberwithin{theorem}{section}

\begin{document}

\title[On rates in approximation theory for operator semigroups]
{On convergence rates in approximation theory for operator semigroups}

\author{Alexander Gomilko}
\address{Faculty of Mathematics and Computer Science\\
Nicolaus Copernicus University\\
ul. Chopina 12/18\\
87-100 Toru\'n, Poland
}

\email{gomilko@mat.umk.pl}

\author{Yuri Tomilov}
\address{Faculty of Mathematics and Computer Science\\
Nicolaus Copernicus University\\
ul. Chopina 12/18\\
87-100 Toru\'n, Poland\\
and
Institute of Mathematics\\
Polish Academy of Sciences\\
\'Sniadeckich 8\\
00-956 Warszawa, Poland
}

\email{tomilov@mat.unk.pl}

\thanks{This work was completed with the support of the NCN grant
 DEC-2011/03/B/ST1/00407.}



\date{\today}

\begin{abstract}
We create a new, functional calculus, approach to approximation of $C_0$-semigroups
on Banach spaces.  As an application of this approach, we obtain optimal convergence rates
in classical approximation formulas for $C_0$-semigroups.
In fact, our methods allow one to derive a number of similar formulas
and equip them with sharp convergence rates. As a byproduct, we prove a new interpolation principle leading
to efficient norm estimates in the Banach algebra of Laplace transforms of bounded measures
on the semi-axis.

\end{abstract}

\maketitle

\section{Introduction}

Approximation theory is a classical chapter in the theory of
$C_0$-semigroups with various applications to PDEs and their
numerical analysis. By approximating a $C_0$-semigroup with
exponentials of bounded operators or with rational functions of
its generator one can often reduce the study of a difficult
problem to a simpler one (otherwise intractable). An instance of
such an approach is the famous Hille-Yosida generation theorem
where Yosida's approximation arises.

The two core results in approximation theory, the Trotter-Kato
theorem and the Chernoff product formula, proved to be very
helpful in many areas of analysis, including differential
operators, mathematical physics and probability theory. The
following particular cases of the  Trotter-Kato theorem
representing different approaches to semigroup approximation
became  well-known and pave their way to the most of books on
semigroup theory, see e.g. \cite[Chapters III.4, III.5]{EngNag},
\cite[Chapters 1.7, 1.8]{Go85}, \cite[Chapters 3.4--3.6]{Pa83},
\cite[Chapter 5]{Cast}.

\begin{thm}\label{introapprox}
Let $(e^{-tA})_{t \ge 0}$ be a bounded $C_0$-semigroup on a Banach
space $X.$ Then the following statements hold.
\begin{itemize}
\item [a)] \, {\rm [Yosida's approximation] } For every $x \in X,$
\[
e^{-tA}x=\lim_{n \to \infty} e^{-n t A(n+A)^{-1}}x
\]
uniformly in $t$ from compacts in $\mathbb R_+.$
\item [b)]\, {\rm [Dunford-Segal approximation]}  For every $x \in X,$
\[
e^{-tA}x =\lim_{n \to \infty} e^{-nt(1-e^{-A/n})}x
\]
uniformly in $t$ from compacts in $\mathbb R_+.$
\item [c)]\, {\rm [Euler's approximation]} For every $x \in X,$
\[
e^{-tA}x =\lim_{n \to \infty} (1+tA/n)^{-n}x
\]
uniformly in $t$ from compacts in $\mathbb R_+.$
\end{itemize}
\end{thm}
(The name for the approximation formula in b) is not
well-established, although some authors use this terminology. We
find it natural too since the formula was introduced for the first
time by Dunford and Segal in \cite{DuSe46}.)

While theorems on semigroup approximation are very useful, with a few exceptions, they still have a merely qualitative
character and the natural problem of finding  optimal rates of
approximation remains open.
The aim of our paper is to fill this gap.

The approximations introduced in Theorem \ref{introapprox} will be of primary importance for us.
So, before describing our approach, we give a short account of known rate estimates for these approximations.
Among the three,
Euler's formula attracted most of
attention and relevant results can be summarized as follows.
\begin{thm}\label{survey}
Let $-A$ be the generator of a bounded $C_0$-semigroup  $(e^{-tA})_{t\ge 0}$ on a Banach space $X.$
Then the following hold.
\begin{itemize}
\item [(i)] \cite[Theorem 4]{BT79} There exists $c >0$ such that for all $n \in \mathbb N$ and $t>0,$
\begin{equation*}
\|e^{-tA}x - (1+tA/n)^{-n}x\|\le c
\left(\frac{t}{\sqrt{n}}\right)^2 \, \|A^2x\|,\qquad x\in {\rm
dom} \,(A^2);
\end{equation*}
\item [(ii)] \cite[Theorem 1.7]{Flory} There exists $c> 0$ such that for all $n \in \mathbb N$ and $t>0,$
\begin{equation*}
\|e^{-tA}x- (1+tA/n)^{-n}x\|\le c\frac{t}{\sqrt{n}}\,
\|Ax\|,\qquad x\in {\rm dom}\,(A);
\end{equation*}
\item [(iii)]\cite[Corollary 4.4]{Kovacs07} There exists $c>0$ such that for all $n \in \mathbb N,$  $t>0$ and
$0< \alpha\le 2,$
\begin{equation*}\label{inter}
\|e^{-tA}x- (1+tA/n)^{-n}x\|\le c
\left(\frac{t}{\sqrt{n}}\right)^{\alpha} \|x\|_{\alpha,2,\infty},
\quad x\in X_{\alpha,2,\infty},
\end{equation*}
where the Banach space $X_{\alpha,2,\infty}$ (called a Favard space) is defined as
\[
X_{\alpha,2,\infty}:=\left\{x\in X:\, \|x\|_{\alpha,2,\infty}:=
\|x\|+\sup_{t>0}\,\frac{\|(e^{-tA}-I)^2x\|}{t^\alpha}<\infty\right\}.
\]
\end{itemize}
\end{thm}
Recall that
$
{\rm dom}\,(A^\alpha)$ is embedded continuously in $X_{\alpha,2,\infty},$ $\alpha\in (0,2),$ but
there are examples (see e.g.
\cite[p. 340]{Komatsu1966})
showing the the inclusion ${\rm dom}\,(A^\alpha) \subset X_{\alpha,2,\infty}$ is in general strict.
The results similar to Theorem \ref{survey} were also obtained in \cite{Hassan}. Note however that
in  \cite{BT79}, \cite{Flory},  \cite{Hassan} and \cite{Kovacs07} rational approximations
more general than Euler's formula were studied,
and Theorem \ref{survey} is a partial case of more general statements proved there.

Much less is known about Yosida's and Dunford-Segal approximation formulas
for bounded $C_0$-semigroups on Banach spaces. Some partial results for analytic semigroups  can be found in
\cite{Ar12} and \cite{Vi1} (see also \cite{Vit}). To the best of our knowledge, there are no papers
 devoted to convergence rates in Yosida's and Dunford-Segal formulas on our level of generality.

In this paper we propose a general approach to approximation
formulas for $C_0$-semigroups on Banach spaces (as in Theorem \ref{introapprox} and similar ones). Our basic
observation is that behind each of the approximation formulas
in Theorem \ref{introapprox} there is a specific Bernstein function and one may look at
approximation issues through the lens of asymptotic relations
\begin{equation}
e^{-nt\varphi(z/n)} \to e^{-tz}, \qquad n \to \infty, \,\, z>0,
\label{a1}
\end{equation}
or
\begin{equation}
 e^{-n\varphi(tz/n)} \to
e^{-tz}, \qquad n \to \infty,\,\, z>0, \label{a2}
\end{equation}
with $\varphi$ being a Bernstein function. In
particular, Yosida's and Dunford-Segal approximations arise when
one puts $\varphi(z)=z/(z+1)$ and $\varphi(z)=1-e^{-z}$ in
\eqref{a1}, respectively, and Euler's approximation corresponds to the choice $\varphi(z)=\log(1+z)$ in \eqref{a2}.
By the Hille-Phillips functional calculus
machinery and subordination, approximations \eqref{a1} and \eqref{a2} give rise to their respective operator
versions
\begin{equation}
e^{-nt\varphi(A/n)} \to e^{-tA}, \qquad n \to
\infty, \label{a11}
\end{equation}
and
\begin{equation}
 e^{-n\varphi(tA/n)} \to
e^{-tA}, \qquad n \to \infty,\label{a21}
\end{equation}
where $-A$ is the generator of a bounded $C_0$-semigroup
$(e^{-tA})_{t \ge 0}$ on a Banach space $X,$ and convergence takes
place in the strong topology of $X.$ Our approach extends the
functional calculus approach to approximation issues started in
the foundational papers  \cite{BT79} and \cite{HK79}, and
developed further in many subsequent articles. For comparatively
recent contributions to this area, see \cite{Flory}, \cite{EgRo}, \cite{Hassan},
and \cite{Ne13}.

Theorem \ref{survey}  suggests that it is natural to study a semigroup approximation formula by
restricting it to the domains of powers of generators, and to hope
that smoothness of the elements will correspond to a certain rate
of convergence in \eqref{a11} and \eqref{a21}.
The present paper develops this observation in a comprehensive way.
Our  Bersntein
functions setting allows one to unify the approximations formulas for  $C_0$-semigroups and
to \emph{equip them with  rates} which, moreover, are optimal under natural spectral assumptions.
As a result,  we quantify classical approximation formulas in Theorem \ref{apprintro1},
restricted to the domains of fractional powers of $A,$ thus extending a number of known results and
proving several new ones. In particular, the following is true.
\begin{thm}\label{apprintro1}
Let  $-A$ be the generator of a bounded $C_0$-semigroup
$(e^{-tA})_{t \ge 0}$ on a Banach space $X,$ and let $\alpha \in
(0,2].$ Then for all $x\in \dom(A^\alpha),\quad t>0,$ and $n \in
\mathbb N,$
\begin{itemize}
\item [a)] \, {\rm [Yosida's approximation]}
\[
\|e^{-tA}x-e^{-n t A(n+A)^{-1}}x\| \le 16 M
\left(\frac{t}{n}\right)^{\alpha/2} \|A^\alpha x\|;
\]
\item [b)]\, {\rm [Dunford-Segal approximation]}
\[
\|e^{-tA}x-e^{-nt(1-e^{-A/n})}x\| \le 8 M
\left(\frac{t}{n}\right)^{\alpha/2} \|A^\alpha x\|;
\]

\item [c)]\, {\rm [Euler's approximation]}
\[
\|e^{-tA}x - (1+ tA/n)^{-n}x\|\le 8
M\left(\frac{t^2}{n}\right)^{\alpha/2} \|A^\alpha x\|.
\]
\end{itemize}
where $M:=\sup_{t \ge0} \|e^{-tA}\|.$
\end{thm}
Theorem \ref{apprintro1},c) was obtained in \cite{GTeuler} with a different proof and a better constant.
Note that  the domain  $(0,2]$ for $\alpha$ in Theorem \ref{apprintro1} cannot in general be enlarged,
see Remark \ref{rangealpha} below.

Another advantage of putting approximation theory into a Bernstein
functions framework is that such a setting reduces the study of
rates to norm estimates in the algebra $\Wip(\mathbb C_+)$ of
Laplace transforms of bounded Radon measures on the semi-axis. In
turn, the estimates in $\Wip(\mathbb C_+)$ can be done in an
elegant and transparent way by means of a new interpolation
principle in $\Wip(\mathbb C_+),$ Theorem \ref{FP} below.

Moreover,  invoking a certain matrix construction and yet another interpolation result, we improve convergence rates in Theorem \ref{apprintro1} if $(e^{-tA})_{t \ge 0}$ is a  bounded analytic
$C_0$-semigroup.
Namely, the following result holds.
\begin{thm}\label{introanal}
Let $(e^{-tA})_{t \ge 0}$ be a  bounded analytic $C_0$-semigroup
on $X,$ with generator $-A$. Then there exists $C>0$ such that for
every $\alpha \in [0,1],$ \begin{itemize}
\item [a)] \, {\rm [Yosida's approximation]}
\[
\|e^{-tA}x-e^{-n t A(n+A)^{-1}}x\| \le C
(nt^{1-\alpha})^{-1} \|A^\alpha x\|;
\]
\item [b)]\, {\rm [Dunford-Segal approximation]}
\[
\|e^{-tA}x-e^{-nt(1-e^{-A/n})}x\| \le
C(nt^{1-\alpha})^{-1} \|A^\alpha x\|;
\]

\item [c)]\, {\rm [Euler's approximation]}
\[
\|e^{-tA}x - (1+t/nA)^{-n}x\|\le
 C(nt^{-\alpha})^{-1} \|A^\alpha x\|,
\]
\end{itemize}
for all
$t
>0, n \in \mathbb N$ and $x \in \dom (A^{\alpha}).$
\end{thm}

In particular, Theorem \ref{introanal},c) extends well-known
results from \cite{LTW91} and \cite{CLPT93} on convergence rates
in  Euler's formula for bounded analytic semigroups restricted to
domains of integer powers of their generators, see also
\cite{Flory}, \cite{JaNe12}, \cite{Vi1} and \cite{Za08}.
  (One may also consult
\cite{BP04} and \cite{Pa04}
where the results from \cite{LTW91} and \cite{CLPT93} were  reproved in an alternative way.)
Theorem \ref{introanal}, a) and b) improves substantially the results from \cite{Ar12} and \cite{Vi2}.

\section{Notations}

For a closed linear operator $A$ on a complex Banach space $X$ we
denote by $\dom(A),$ $\ran(A),$ $\rho(A)$ and $\sigma(A)$ the
{\em domain}, the {\em range}, the {\em resolvent set} and the {\em
spectrum} of $A$, respectively, and let  $\cls{\ran}(A)$ stand for the norm-closure of the range of $A$.
We will write the Lebesgue integral $\int_{(0,\infty)}$ as
$\int_{0+}^{\infty}$. (Such a notation is established in
\cite{SchilSonVon2010} and, as far as we rely on
\cite{SchilSonVon2010} essentially, we decided to use this
notation too.) With a slight abuse of notation, we use  $z^a f$ to
denote the function $z\to z^a f(z)$ as an element of a function
space.

 Finally, we let
\[
\mathbb C_{+}=\{z\in \C:\,{\rm Re}\,z>0\},\quad \mathbb R_+=[0,\infty).
\]

\section{Completely monotone, Stieltjes and  Bernstein functions}

In this section we collect  notions and facts from function theory needed in the sequel.

Let ${\rm M}(\mathbb R_+)$ be a Banach algebra of bounded Radon measures on $\mathbb R_+.$
It will be convenient to work with the image of ${\rm M}(\mathbb R_+)$ under the Laplace transform.
Recall  that the {\em Laplace transform} of  $\mu \in {\rm M}(\mathbb R_+)$  is given by
\[ (\mathcal L\mu)(z) := \int_0^{\infty}e^{-sz} \, \mu(d{s}),
\qquad  z \in \mathbb C_+.
\]
If
\begin{eqnarray*}
{\rm A}^1_+(\mathbb C_+) &:=& \{ \mathcal L\mu : \mu \in {\rm M}(\mathbb R_+)\}\\
\|\mathcal L \mu\|_{{\rm A}^1_+(\mathbb C_+)} &:=& \|\mu\|_{{\rm M}(\mathbb R_+)} = |{\mu}|(\mathbb R_+),
\end{eqnarray*}
where $|\mu|(\mathbb R_+)$ denotes the total variation of $\mu$ on
$\mathbb R_+,$ then it is easy to prove that $(\Wip(\mathbb C_+),\|\cdot\|_{\Wip(\mathbb
C_+)})$ is a commutative Banach algebra with pointwise
multiplication, and the Laplace transform
\[ \mathcal L : {\rm M}(\mathbb R_+) \mapsto {\rm A}^1_+(\mathbb C_+)
\]
is an isometric isomorphism.

The class of completely monotone functions can be considered as a generalization of
$ \Wip(\mathbb C_+)$ to the setting of Laplace transforms of, in
general, unbounded positive measures. Recall that  a function $f\in
\Ce^\infty(0, \infty)$ is said to be completely monotone if
\[
f(t)\geq 0\quad \mbox{and}\quad (-1)^n \frac{d^n
f(t)}{d{t}^n}\ge 0 \qquad \text{for all $n \in \mathbb N$ and $
t > 0$}.
\]
By the Bernstein theorem \cite[Theorem 1.4]{SchilSonVon2010} any
complete monotone function $f$ is the Laplace transform of a
unique positive Radon measure $\nu$ on $\mathbb R_+$, i. e. for
all $z>0$,
\[
f(z)=\int_{0}^{\infty} e^{-z t}\,\nu(dt).
\]
Note that any bounded completely monotone function $f$ belongs to $\Wip(\C_+)$ by
Fatou's lemma. Moreover, if we denote $f_\tau(\cdot)=f(\tau \cdot), \tau >0,$ then $f_\tau \in \Wip(\C_+)$ and
\begin{equation}\label{normcomp}
\|f_\tau\|_{\Wip(\C_+)}=f(0+), \qquad \tau >0.
\end{equation}

The notion of completely monotone function is closely related to
the notion of Bernstein function which will be central for the
studies in this paper. A positive function $\varphi \in \Ce^\infty(0,
\infty)$ is called {\em Bernstein function} if its derivative is
completely monotone. By \cite[Theorem 3.2]{SchilSonVon2010},
$\varphi$ is Bernstein if and only if there exist  $a, b\geq 0$
and a positive Radon measure $\mu$ on $(0,\infty)$ satisfying
\begin{equation*}
\int_{0+}^\infty\frac{s}{1+s}\,\mu(d{s})<\infty \label{mu}
\end{equation*}
such that
\begin{equation}\label{defBF}
\varphi(z)=a+bz+\int_{0+}^\infty (1-e^{-zs})\mu(d{s}), \qquad z>0.
\end{equation}
The formula \eqref{defBF} is called the Levy-Khintchine
representation of $\varphi.$ The triple $(a, b, \mu)$ is uniquely
determined by the corresponding Bernstein function $\varphi$ and
is called the Levi-Khintchine triple. Thus we will write
occasionally $\varphi \sim (a,b, \mu).$ Every Bernstein function
extends analytically to $\mathbb C_+$ and continuously to
$\overline{\mathbb C_+}.$ In
 the following a Bernstein function will be identified with its continuous extension to $\overline{\mathbb C_+}.$

The class of Bernstein functions can also be viewed in terms of
the algebra $\Wip(\mathbb C_+)$ as the following statement shows,
see \cite[Lemma 2.5]{GoHaTo12}.
\begin{lemma}\label{exthil}
Every Bernstein function $\varphi$ can be written in the form
\[ \varphi(z) = \psi_1(z) + z\, \psi_2(z), \qquad z \ge 0,
\]
where $\psi_1, \psi_2 \in \Wip(\C_+)$.
\end{lemma}

From the definition it follows that a Bernstein function $\varphi$  is increasing,  and if $\varphi \sim (a,b, \mu)$ then
\begin{equation}\label{ab}
a=\varphi(0)\quad \text{and} \quad b=\lim_{t \to
\infty}\frac{\varphi(t)}{t}.
\end{equation}

Since we will often deal with bounded Bernstein functions, we collect several simple
characterizations of such functions in
the following lemma.
\begin{lemma}\label{boundedbe}
Let $\varphi$ be a Bernstein function. Then the following
statements are equivalent.
\begin{itemize}
\item [(i)]$\varphi$ is bounded on $\mathbb R_+.$
\item [(ii)] $\varphi'
\in L^1(0,\infty).$
\item [(iii)] $\varphi$ has a Levy-Hintchine representation of the form
$(a,0,\mu)$ with $\mu$ being a bounded positive Radon measure on
$(0,\infty).$
\end{itemize}
\end{lemma}
For the proof it suffices to note that the equivalence of (i) and (ii) follows from the
positivity of $\varphi'$ and Newton-Leibnitz formula, while
the equivalence of (i) and (iii) is proved in \cite[Corollary 3.8, (v)]{SchilSonVon2010}.
 (It is  a direct consequence of  Fatou's lemma and the Levy-Hintchine representation \eqref{defBF}.)

Moreover, from \eqref{defBF} it follows that if $\varphi$ is a
bounded Bernstein function then $b=0$ and
\[
\varphi(\infty)-\varphi(0):=\lim_{z
\to+\infty}\,\varphi(z)-a=\int_{0+}^\infty \mu(ds)<\infty,
\]
hence $\varphi\in \Wip(\C_+)$ and
\begin{equation}\label{normbern}
\|\varphi\|_{\Wip(\C_+)}=a+2\int_{0+}^\infty
\mu(ds)=2\varphi(\infty)-\varphi(0).
\end{equation}

It will be crucial for the sequel that if $f$ is completely
monotone and $\varphi$ is Bernstein function, then $f\circ
\varphi$ is completely monotone by \cite[Theorem
3.6]{SchilSonVon2010}. If, in addition, $f$ is bounded, then $f\circ
\varphi$ is bounded and completely monotone, and Fatou's theorem
yields the following estimate for the $\Wip(\mathbb C_+)$-norm of
$f\circ \varphi$:
\begin{equation}\label{Wip}
\|f\circ \varphi\|_{\Wip(\C_+)}=(f\circ \varphi)(0+) \le
\|f\|_{\Wip(\C_+)}.
\end{equation}
Note finally that $f$ and $\varphi$ are Bernstein functions then
$f \circ \varphi$ is Bernstein as well. For these properties of
compositions see \cite[Thereom 3.6 and Corollary
3.7]{SchilSonVon2010}.

One of the most important properties of Bernstein functions is
that their exponentials one-to-one correspond to convolution
semigroups of subprobability measures and the correspondence is
given by the Laplace transform. The following statement can be
found e.g. in \cite[Theorem 5.2]{SchilSonVon2010}. Recall that a
family of Radon measures $(\mu_t)_{t\ge 0}$ on $\mathbb R_+$ is
called a vaguely continuous convolution semigroup of
subprobability measures if $\mu_t(\mathbb R_+)\le 1$ for all $t\ge
0, \mu_{t+s}=\mu_t*\mu_s$ for all $t,s\ge 0$ and
vague-$\lim_{t\to0+}\mu_t=\delta_0,$ where $\delta_0$ stands for
the Dirac measure at zero.
\begin{thm}\label{subordf}
The function $\varphi$ is Bernstein if and only if there exists a
vaguely continuous semigroup $(\mu_t)_{t \ge 0}$ of subprobability
measures on $\mathbb R_+$ such that
\begin{equation}\label{lapber}
({\mathcal L} \mu_t)(z)=e^{-t\varphi(z)}, \qquad z \in \mathbb
C_+,
\end{equation}
for all $t \ge 0.$
\end{thm}

The  subclass of Bernstein functions introduced below will be fundamental for our studies.
\begin{defn} Let
\begin{equation}
\Phi:=\{\varphi \,\, \text{is Bernstein}: \, \, \varphi(0)=0,\quad
\varphi'(0+)=1,\quad |\varphi''(0+)|<\infty\}.\label{3cond}
\end{equation}
\end{defn}
\begin{remark}\label{eqCond}
If $\varphi$ has the Levy-Hintchine representation $(0,0,\mu)$
then the assumptions (\ref{3cond}) are equivalent to
\begin{equation}\label{Equiv}
\int_{0+}^\infty s\, \mu(ds)=1,\qquad \int_{0+}^\infty
s^2\, \mu(ds)<\infty.
\end{equation}
\end{remark}
It will be essential for us that the functions
$\varphi_1(z)=z/(z+1),$ $\varphi_2(z)=1-e^{-z}$ and
$\varphi_3(z)=\log(1+z)$ are Bernstein (see \cite[Chapter 15]{SchilSonVon2010}), and moreover
$\{\varphi_1,\varphi_2,\varphi_3\} \subset \Phi.$

An important subclass of completely monotone functions is formed by Stieltjes functions.
A function $f : (0, \infty) \to \R_+$ is called a {\em Stieltjes}
function if it can be written as
\begin{equation}\label{hpfc.e.stieltjes}
f(z) = \frac{a}{z} + b + \int_{0+}^\infty \frac{\nu(\ud{s})}{z +
s}, \qquad z > 0,
\end{equation}
where $a, b \ge 0$ and $\nu$ is a positive Radon measure on $(0,
\infty)$ satisfying
\[ \int_{0+}^{\infty} \frac{\nu(\ud{s})}{1 + s} < \infty.
\]
In this case, $\mu$ is called a {\em Stieltjes measure} and
\eqref{hpfc.e.stieltjes} is called the {\em Stieltjes
representation} for $f$, since such a representation is unique.
Moreover, every nonzero Stieltjes function $f$ is completely
monotone, and the functions $zf$ and $1/f$ are Bernstein, see
\cite[Chapters 6-7]{SchilSonVon2010} for more details. It follows
from the definition, that the functions $z^{-\alpha}, \alpha \in
[0,1],$ are Stieltjes.

We will also need a generalization of Stieltjes functions called
sometimes generalized Stieltjes functions. Recall that a function
$f:(0,\infty)\mapsto \mathbb R_+$ is {\it generalized Stieltjes}
of order $\alpha>0$ if it can be written as
\begin{equation}\label{stilrep}
f(z)= \frac{a}{z^{\alpha}} +b +\int_{0+}^\infty
\frac{\nu(ds)}{(z+s)^\alpha}, \qquad z>0,
\end{equation}
where $a,b\ge 0$ and $\nu$ is a positive Radon measure on
$(0,\infty)$ satisfying
\begin{equation*}
\int_{0+}^\infty\frac{\nu(ds)}{(1+s)^\alpha}<\infty.
\end{equation*}
Observe that if $f$ is generalized Stieltjes (of any positive order), then $f$ admits an (unique) analytic
extension to $\mathbb C \setminus (-\infty,0]$
which will be identified with $f$ and denoted by the same symbol.
The class of generalized Stieltjes functions of order $\alpha$ will be denoted by $\mathcal{S}_\alpha.$
In this terminology, Stieltjes functions constitute precisely  the class $\mathcal{S}_1$ of generalized
Stieltjes functions of order $1.$
For a thorough treatment of generalized Stieltjes functions one may consult  \cite{Karp}.

\subsection{Functional Calculus}\label{fcc}

Let $-A$ be the generator of a bounded $C_0$-semigroup
$(e^{-tA})_{t\ge 0}$ on a Banach space $X$. Recall that the mapping
\[ g = \mathcal L \mu  \quad
\mapsto \quad g(A) := \int_{\R_+} e^{-tA}\, \mu(\ud{t})
\]
(with a strong Bochner integral in the definition of $g(A)$)
defines a continuous algebra homomorphism from $\Wip(\C_+)$ into
the Banach space of bounded linear operators on $X$ satisfying
\begin{equation}\label{hille-pillips}
 \norm{g(A)} \le (\sup_{t \ge 0} \norm{e^{-tA}} ) \, \norm{g}_{\Wip(\mathbb C_+)},
\qquad g\in \Wip(\C_+).
\end{equation}
This homomorphism is called the {\em Hille--Phillips} (HP)
functional calculus for $A$, see \cite[Chapter XV]{HilPhil}. It
admits an extension to a class of functions larger than
$\Wip(\C_+)$. This extension is constructed in the following way.
If $f: \C_+ \to \C$ is holomorphic such that there exists  $e\in
\Wip(\C_+)$ with $ef \in \Wip(\C_+)$ and the operator $e(A)$ is
injective, then we set
\begin{eqnarray*}
\dom (f(A))&:=&\{x \in X: (ef)(A)x \in \ran(e(A)) \}\\
 f(A) &:=& e(A)^{-1} \, (ef)(A).
\end{eqnarray*}
In this case $f$ is called {\em regularizable}, and $e$ is called
a {\em regularizer} for $f$. Such a definition of $f(A)$ does not
depend on the regularizer $e$ and $f(A)$ is a closed
 operator on $X$. Moreover, the set of all
regularizable functions $f$ forms an algebra ${\mathcal A}$ (depending on $A$).
(See e.g. \cite[p. 4-5]{Ha06}.)
The mapping
\[ {\mathcal A} \ni f \tpfeil f(A)
\]
from ${\mathcal A}$ into the set of all closed operators on $X$ is
called the {\em extended Hille--Phillips (HP) functional calculus} for $A$.
The calculus possess
a number of natural properties, and we refer to
\cite[Chapter 1]{Ha06} for the most of them. In this paper, the next product rule will be important:
 if $f$ is regularizable and $g\in \Wip(\C_+)$,
then
\begin{equation}\label{hpfc.e.prod}
 g(A) f(A) \subseteq f(A) g(A) = (fg)(A),
\end{equation}
where we take the natural domain for a product
of operators. If $f \in \Wip (\mathbb C_+),$ then $f(A)$ is bounded and clearly $g(A)f(A)=(fg)(A).$

The following well-known spectral mapping (inclusion) theorem for
the HP functional calculus will also be crucial. See e.g.
\cite[Theorem 16.3.5]{HilPhil} or \cite[Theorem 2.2]{GoHaTo12} for
its proof.
\begin{thm}\label{mapping}
Let $f \in \Wip (\mathbb C_+)$ and  $-A$ be the generator of a bounded $C_0$-semigroup on a Banach space $X.$
Then
\begin{equation}\label{spmapping}
\{f(\lambda): \lambda \in \sigma(A)\} \subset \sigma(f(A)).
\end{equation}
\end{thm}

Bernstein functions are, in fact,  part of the extended
HP functional calculus. Indeed, as a consequence of Lemma
\ref{exthil}, we obtain that every Bernstein function is
regularizable by any of the functions  $e_{\lambda}(z)=(\lambda
+z)^{-1}$, $\re \lambda
> 0$.
Moreover, the next operator-valued analogue of \eqref{defBF}
holds,
 see e.g. \cite[Corollary 2.6]{GoHaTo12} and \cite[Corollary 12.21]{SchilSonVon2010}.
\begin{thm}\label{hpfc.c.bf-fc}
Let $-A$ be the generator of a bounded $C_0$-semigroup
$(e^{-tA})_{t\ge 0}$ on $X$, and let $\varphi \sim (a,b,\mu)$ be a
Bernstein function. Then $\varphi(A)$ is defined in the extended
HP functional calculus. Moreover, $\dom(A) \subseteq
\dom(\varphi(A))$ and
\begin{equation}\label{phillips}
\varphi(A)x = ax + bAx + \int_{0+}^\infty (1 - e^{-tA})x \,
\mu(d{t})
\end{equation}
for each $x\in \dom(A),$ and $\dom(A)$ is a core for $\varphi(A)$.
\end{thm}

Similarly, there is the next operator-valued counterpart of the
representation \eqref{stilrep} for Stieltjes functions, see e.g.
\cite[Theorem 2.5,{\it (ii)}]{GoTo}.
\begin{prop}\label{resolventst}
Let  $-A$ be the generator of a bounded $C_0$-semigroup on a Banach space $X$.
 If $f \sim (0,0,\mu)$ is  a Stieltjes function and $A$ has dense range, then
$f$ belongs to the extended HP calculus and
\begin{equation}\label{stiltcalc}
f(A)x=\int_{0+}^{\infty}(s+A)^{-1}x \,\mu(ds)
\end{equation}
for every $x \in \ran (A).$ Moreover, $\ran(A)$ is a core for
$f(A).$
\end{prop}
It would be instructive to recall that, by the mean ergodic theorem, $A$ is injective if it has dense range.
Thus, Proposition \ref{resolventst} requires, in particular, the injectivity of $A.$

Finally, we note that if $-A$ generates a bounded $C_0$-semigroup
on $X$, then for any Bernstein function $\varphi$ the operator
$-\varphi(A)$ generates a bounded $C_0$-semigroup on $X$ as well
and the latter semigroup can be represented in terms of $\varphi$
and $(e^{-tA})_{t \ge 0}$ as the formula \eqref{lapber} suggests,
see e.g. \cite[Theorem 12.6]{SchilSonVon2010} which we quote below.
\begin{thm}\label{subord}
Let $-A$ be the generator of a bounded $C_0$-semigroup
$(e^{-tA})_{t \ge 0}$ on a Banach space $X$, and let $\varphi$ be
a Bernstein function with the corresponding vaguely continuous
convolution semigroup $(\mu_t)_{t\ge 0}$  of subprobability
measures on $\mathbb R_+.$ Then the formula
\begin{equation}
e^{-t\varphi(A)}:=\int_{0}^{\infty} e^{-sA}\, \mu_t(ds), \qquad t
\ge 0,
\end{equation}
defines a bounded $C_0$-semigroup $(e^{-t\varphi(A)})_{t \ge 0}$
on $X$ with generator $-\varphi(A),$ where $\varphi(A)$ is given by
\eqref{phillips} on $\dom(A)$.
\end{thm}

\section{Interpolation estimates in $\Wip(\C_+)$}

We start with with a lemma providing estimates for the
$\Wip(\C_+)$-norms of a pair of functions used in the proof of our
interpolation principle for $\Wip(\C_+),$ Theorem \ref{FP} below.
\begin{lemma}\label{Propbeta}
Let
\begin{equation}\label{betaW}
u_{\beta,\tau}(z):=\left(\frac{z}{z+\tau}\right)^\beta, \quad v_{\beta,\tau}(z):=\frac{1}{(z+\tau)^\beta}, \qquad z  \in \mathbb C_+,
\end{equation}
for  $\beta>0$ and $\tau>0.$
Then $u_{\beta,\tau}, v_{\beta, \tau} \in \Wip(\C_+)$ for all $\beta,\tau>0$ and
\begin{align}
2^{\beta+1} \ge \|u_{\beta,\tau}\|_{\Wip(\C_+)}\ge& 1 ,\label{inequ} \\
 \|v_{\beta,\tau}\|_{\Wip(\C_+)} =&\frac{1}{\tau^\beta}.
\end{align}
\end{lemma}

\begin{proof}
Let $\tau >0$ be fixed. Note that $ u_{1,\tau} $ is a
Bernstein function. If $\beta \in (0,1),$ then the function $z \mapsto
z^\beta$ is Bernstein, and $u_{\beta,\tau}=(u_{1,\tau})^\beta$ is
a  Bernstein function too as a composition of Bernstein functions. Moreover, it is bounded since $
u_{1,\tau} $ is bounded. Thus $u_{\beta,\tau}\in \Wip(\C_+)$  and
by \eqref{normbern},
\begin{equation}\label{betebeta}
\|u_{\beta,\tau}\|_{\Wip(\C_+)}=-u_{\beta,\tau}(0)+2u_{\beta,\tau}(\infty)=2,\quad \beta\in (0,1].
\end{equation}
If $\beta>1$ then
\[
\beta=n+\alpha,\qquad n\in  \N\cup \{0\},\quad \alpha \in [0,1),
\]
and
\begin{equation*}\label{Cn}
\|u_{\beta,\tau}\|_{\Wip(\C_+)}\le \|u_{n,\tau}\|_{\Wip(\C_+)}\|u_{\alpha,\tau}\|_{\Wip(\C_+)} \le 2^{n+1}\le 2^{\beta+1}.
\end{equation*}
Moreover,
 $$\|u_{\beta,\tau}\|_{\Wip(\C_+)} \ge \sup_{z \in \mathbb
C_+}|u_{\beta,\tau}(z)| \ge \lim_{z\to\infty}\,|u_{\beta,\tau}(z)|=1. $$
This gives the second estimate in \eqref{inequ}.

To prove the assertion about $v_{\beta,\tau},$ note that
\begin{equation}\label{Exp}
v_{\beta,\tau}(z) =\frac{1}{\Gamma(\beta)}\int_0^\infty
e^{-zs}e^{-\tau s} s^{\beta-1}\,ds, \qquad z \in \mathbb C_+,
\end{equation}
where $\Gamma$ is the gamma-function.
 Thus $v_{\beta,\tau}$ is a bounded completely monotone
function for all $\beta >0$ and  $\tau> 0.$ Hence
\begin{equation*}
\|v_{\beta,\tau}\|_{\Wip(\C_+)}=v_{\beta,\tau}(0+)=\frac{1}{\tau^\beta}.
\end{equation*}
\end{proof}
Define now
\begin{equation}\label{denote}
c_\beta:=\|u_{\beta,\tau}\|_{\Wip(\C_+)}
=\|u_{\beta,1}\|_{\Wip(\C_+)},\quad \beta>0.
\end{equation}
Note that the bounds for $c_\beta$ in \eqref{inequ} are far from
being optimal and better estimates requiring however much more
involved arguments are given in Appendix. For instance, since
\[
u_{2,1}(z)=\left(1-\frac{1}{z+1}\right)^2=
1+\int_0^\infty e^{-zs} e^{-s}(s-2)\,ds,
\]
it follows that
\begin{equation}\label{c2}
c_2=\|u_{2,1}\|_{\Wip(\C_+)}=1+\int_0^\infty  e^{-s} | s-2|\,ds
=2(1+e^{-2})\le \frac{5}{2}.
\end{equation}
On the other hand, \eqref{inequ} is completely sufficient for our purposes. The inequality \eqref{c2}
will be used in the next section.

Our arguments depend essentially on sharp estimates of $\|F
f\|_{\Wip(\C_+)}$ for $F\in \Wip(\C_+)$ and $f \in \mathcal{S}_\beta.$
The next result of independent interest yields  bounds for $\|F
f\|_{\Wip(\C_+)}$ in terms of $\|F\|_{\Wip(\C_+)}$ and
$\|z^{-\beta} F\|_{\Wip(\C_+)}$ thus allowing us to simplify
technical details and to make our presentation more transparent.

\begin{thm}\label{FP}
Let $F\in \Wip(\C_+)$ be such that $z^{-\beta}F \in \Wip(C_+)$ for
some $\beta>0$, and denote
\begin{equation}\label{condition}
\|F\|_{\Wip(\C_+)}=a \qquad \text{and} \qquad \|z^{-\beta}
F\|_{\Wip(\C_+)}=b.
\end{equation}
Then for any $f\in \mathcal{S}_\beta,$
\begin{equation}\label{res}
\|F f\|_{\Wip(\C_+)}\le  2^{\beta}   a f\left(\left({a}/{ b}\right)^{1/\beta}c_\beta^{-1/\beta}\right).
\end{equation}
\end{thm}

\begin{proof}
Fix $\beta >0$ and
consider first $f_0\in
\mathcal{S}_\beta$ of the form
\begin{equation}\label{fnol}
f_0(z)=\int_{0+}^\infty \frac{\nu(d\tau)}{(z+\tau)^\beta}, \qquad
z \in \mathbb C_+,
\end{equation}
where $\mu$ is a Radon measure on $(0,\infty)$ satisfying
\begin{equation*}
\int_{0+}^\infty\frac{\nu(d\tau)}{(1+\tau)^\beta}<\infty.
\end{equation*}
 For $\tau >0$ let
$u_{\beta,\tau}$ and $v_{\beta,\tau}$ be the functions defined in
\eqref{betaW}. Then
\begin{equation}\label{repdelta}
F(z)f_0(z)= \int_{0+}^\infty v_{\beta,\tau}(z) F(z)\,\nu(d\tau).
\end{equation}

 We  use that
$u_{\beta,\tau}, v_{\beta, \tau} \in \Wip(\C_+)$ by Lemma \ref{Propbeta}.
First, note that
\[
\left\|v_{\beta,\tau} F\right\|_{\Wip(\C_+)}\le
\left\|v_{\beta,\tau}\right\|_{\Wip(\C_+)}\left\|
F\right\|_{\Wip(\C_+)}\le \frac{a}{\tau^\beta}.
\]
On the other hand,
\[
\left\|v_{\beta,\tau} F\right\|_{\Wip(\C_+)}
\le \|u_{\beta,\tau}\|_{\Wip(\C_+)}
\| z^{-\beta}\cdot F \|_{\Wip(\C_+)}
\le c_\beta b.
\]
Then using the elementary inequality
\begin{equation}\label{elementary}
\min\left(x^{-\beta}, y^{-\beta}\right)\le \frac{2^\beta}{(x+y)^\beta},\quad x,y>0,
\end{equation}
we obtain
\begin{eqnarray}\label{12a}
\left\|v_{\beta,\tau} F\right\|_{\Wip(\C_+)} &\le& a
\min\left(\tau^{-\beta}, c_\beta b/a\right)
 \\
&\le& \frac{2^\beta a}{((c_\beta b/a)^{-1/\beta}+\tau)^\beta}.
\notag
\end{eqnarray}
Observe further that, in view of \eqref{Exp}, $\tau \mapsto F
v_{\beta,\tau}$ is a continuous, ${\rm A}^1_+(\mathbb C_+)$-valued
function on $(0,\infty).$ Moreover, by \eqref{12a}, the function
$\tau \mapsto \|F v_{\beta,\tau} \|_{{\rm A}^1_+(\mathbb C_+)}$ is
$\nu$-integrable on $(0,\infty).$   Thus the ${\Wip(\C_+)}$-valued
Bochner integral
\begin{equation*}
\int_{0+}^\infty  F v_{\beta,\tau}\,\nu(d\tau)
\end{equation*}
is well-defined. Since point evaluations $f \mapsto f(z) (z \in
\mathbb C_+)$ are bounded linear functionals on $\Wip(\C_+),$ and
separate elements of $\Wip(\C_+),$ \eqref{repdelta} implies that
the integral coincides with $F f_0.$ Hence by \eqref{repdelta},
\eqref{12a}, and a standard  inequality for Bochner integrals we
obtain
\begin{equation}\label{FFF}
\|F f_0\|_{\Wip(\C_+)} \le \int_{0+}^\infty \|F v_{\beta,
\tau}\|_{\Wip(\C_+)}\,\nu(d\tau)\le
  2^{\beta}   a f_0\left(\left({a}/{ b}\right)^{1/\beta}c_\beta^{-1/\beta}\right).
\end{equation}
Let now $f \in S_\beta.$  Then  $f(z) = c_0 z^{-\beta}+c_1 + f_0,$ where $c_0,c_1 \ge 0$ and $f_0$ is given by \eqref{fnol}. Using \eqref{FFF} and and taking into account that $c_\beta \ge 1$, we have
\begin{eqnarray*}
\|Ff\|_{\Wip(\C_+)}&\le& c_0 b + c_1 a + 2^{\beta}   a f_0\left(\left({a}/{b}\right)^{1/\beta}c_\beta^{-1/\beta}\right)\\
&=& 2^{\beta} a \left(c_0 \frac{b}{2^\beta a}+\frac{c_1}{2^\beta}+f_0((a/b)^{1/\beta} c_\beta^{-1/\beta}\right)\\
&\le& 2^{\beta}  a
f\left(\left({a}/{b}\right)^{1/\beta}c_\beta^{-1/\beta}\right).
\end{eqnarray*}
\end{proof}

In our studies of approximation formulas, we will be using the
following straightforward corollary of Theorem \ref{FP} rather
than Theorem  \ref{FP} itself. The corollary clarifies the
interpolation nature of Theorem \ref{FP}.

\begin{cor}\label{coralpha}
Let $F\in \Wip(\C_+)$ be such that
$z^{-\beta}  F \in \Wip(\C_+)$ for some  $\beta>0$, and denote
\begin{equation}\label{condition1}
\|F\|_{\Wip(\C_+)}=a \qquad \text{and} \qquad \|z^{-\beta} F\|_{\Wip(\C_+)}=b.
\end{equation}
Then for every $\alpha\in [0,\beta]$ one has $z^{-\alpha} F \in
\Wip(\C_+)$ and
\begin{equation}\label{res1}
\|z^{-\alpha} F\|_{\Wip(\C_+)}\le  2^{\beta} c_\beta^{\alpha/\beta} a^{1-\alpha/\beta} b^{\alpha/\beta}.
\end{equation}
\end{cor}
For the proof of Corollary \ref{coralpha} it suffices to note that
$z^{-\alpha}\in \mathcal{S}_\beta$ if $\alpha \in [0,\beta]$ (see
\cite[Theorem 3]{Karp}).
\smallskip

 For a
 function $\varphi \in \Phi$ define
\begin{equation}\label{deltaphi}
\Delta^{\varphi}_t(z):=e^{-t\varphi(z)}-e^{-tz}, \qquad z \in \mathbb C_+, \quad t >0.
\end{equation}

We proceed with deriving several consequences of Corollary
\ref{coralpha} which will be instrumental in Sections
\ref{applgeneral} and \ref{secanal}  dealing with convergence
 rates in approximation formulas for $C_0$-semigroups.
The convergence
 rates will follow from estimates of the $\Wip(\mathbb C_+)$-norm
 of $z^{-\alpha} \Delta^{\varphi}_t, \alpha \in [0,2].$ To get
 bounds for $\| z^{-\alpha} \Delta^{\varphi}_t\|_{\Wip(\mathbb
 C_+)}$ using Corollary \ref{coralpha}
 we first compute the $\Wip(\mathbb C_+)$-norm of $z^{-2} \Delta^{\varphi}_t.$
\begin{prop}\label{T1Dop}
Let $\varphi \in \Phi.$ For every  $t >0,$ the function
$z^{-2} \Delta^{\varphi}_t$ belongs to $\Wip(\C_+)$ and
\[
\|z^{-2} \Delta^{\varphi}_t\|_{\Wip(\C_+)}=\frac{t}{2}|\varphi''(0+)|, \qquad t
>0.
\]
\end{prop}

\begin{proof}
Fix $t >0$. By Theorem \ref{subordf}   there exists a convolution semigroup of subprobability  Radon
measures $\nu_t$ on $[0,\infty)$ such that
\begin{equation}\label{subordd}
e^{-t\varphi(z)}=\int_0^\infty e^{-zs}\,\nu_t(ds).
\end{equation}
Since $e^{-t\varphi}$ is completely monotone and $\varphi(0)=0,$ \eqref{normcomp} implies that
\begin{equation}\label{exp1Dop}
\|e^{-t\varphi}\|_{\Wip(\C_{+})}=e^{-t\varphi(0)}=
\int_0^\infty\,\nu_t(ds)=1.
\end{equation}

Let us further show that
\begin{equation}\label{GtSDop}
z^{-2} \Delta^{\varphi}_t(z)=\int_0^\infty e^{-z s}G_t(s)\,ds,
\end{equation}
where
\begin{eqnarray*}
G_t(s)=\chi(t-s)\int_0^s (s-\tau)\, \nu_t(d\tau)
+\chi(s-t)\int_s^\infty (\tau -s)\,\nu_t(d\tau),
\end{eqnarray*}
and $\chi(\cdot)$ stands for the characteristic function of
$\mathbb R_+$. Taking into account \eqref{subordd} and that
\[
\frac{1}{z^2}=\int_0^\infty e^{-z\tau}\tau\,d\tau,\quad z>0,
\]
we have
\begin{eqnarray*}
z^{-2} \Delta^{\varphi}_t(z)&=&\int_0^\infty e^{-z\tau} \tau d\tau \int_0^\infty e^{-zs}
\,\nu_t(ds) -\int_t^\infty e^{-zs}(s-t)\,ds
\\
&=&\int_0^\infty e^{-zs} \,\int_0^s (s-\tau) \,\nu_t(d\tau) \, ds-
\int_t^\infty e^{-zs}(s-t)\,ds\\
&=&\int_{0}^{\infty}e^{-zs} H_t(s)\, ds,
\end{eqnarray*}
where
\[
 H_t(s)= \int_0^s (s-\tau)\, \nu_t(d\tau)-
\chi(s-t)(s-t).
\]
It is clear that if $t \ge s$ then $H_t(s)=G_t(s).$ We prove that
$H_t(s)=G_t(s)$ for $t  < s$ as well. If $s >t,$ then using
\eqref{exp1Dop} and
\[
\int_0^\infty \tau \,\nu_t(d\tau)=-(e^{-t\varphi})'(0+)=t,
\]
we infer that
\begin{eqnarray*}
H_t(s)&=&s\left(\int_0^s\, \nu_t(d\tau)-1\right)+t-\int_0^s \tau \,\nu_t(d\tau)\\
&=&-s\int_s^\infty \,\nu_t(d\tau)+
\int_s^\infty \tau \,\nu_t(d\tau)\\
&=&\int_s^\infty (\tau -s)\,\nu_t(d\tau),
\end{eqnarray*}
hence  $H_t(s)=G_t(s).$ Thus \eqref{GtSDop} holds. Since
$G_t(s)>0,$ $s>0,$ the function $z^{-2} \Delta^{\varphi}_t$ is completely monotone.
Hence, applying Lopital's rule twice, we get
\[
\|z^{-2} \Delta^{\varphi}_t\|_{\Wip(\C_+)}=\lim_{z\to 0+}\,
\frac{e^{-t\varphi(z)}-e^{-tz}}{z^2}= \frac{t}{2}|\varphi''(0+)|.
\]
\end{proof}

 Now we are in position to use Corollary \ref{coralpha} for estimates of
 $\|z^{-\alpha}\Delta^{\varphi}_{t}\|_{\Wip(\C_+)}$ when $\alpha \in [0,2].$
 \begin{cor}\label{corrr}
Let  $\varphi \in \Phi$. Then for every $\alpha \in [0,2],$
\begin{equation}\label{ineqdelta}
 \|z^{-\alpha} \Delta^{\varphi}_{t}\|_{\Wip(\mathbb C_+)} \le 8  \left(t|\varphi''(0+)|\right)^{\alpha/2}.
\end{equation}
\end{cor}
\begin{proof}
First note that
\begin{equation*}
\|\Delta^{\varphi}_{t}\|_{\Wip(\C_+)}\le \|e^{-t\varphi(\cdot)}\|_{\Wip(\C_+)}
+ \|e^{-t\cdot}\|_{\Wip(\C_+)}= 2, \quad \quad t>0.
\end{equation*}
Moreover,   Proposition \ref{T1Dop} yields
$$\|z^{-2}\Delta^{\varphi}_{t}\|_{\Wip(\C_+)}=
\frac{t}{2}|\varphi''(0+)|.$$ Then, using  Corollary
\ref{coralpha} (with $\beta=2$) and the bound for $c_2$ in
\eqref{c2}, we obtain
\begin{equation*}
 \|z^{-\alpha}\Delta^{\varphi}_{t}\|_{\Wip(\C_+)} \le 8  \left(t|\varphi''(0+)|\right)^{\alpha/2}.
\end{equation*}

\end{proof}

For all $n \in \mathbb N$, $t> 0,$ and $z \in \C_+$ define the scaled versions of \eqref{deltaphi}:
\begin{eqnarray}
\Delta^{\varphi}_{t,n}(z):=e^{-nt\varphi(z/n)}-e^{-tz},\label{deltadef}
\end{eqnarray}
and
\begin{eqnarray}
E^{\varphi}_{t,n}(z):=e^{-n\varphi(zt/n)}-e^{-tz} \label{edef},
\end{eqnarray}
and note that
\begin{equation}\label{relation}
\Delta^{\varphi}_{t,n}(z)=\Delta^{\varphi}_{nt}(z/n), \qquad E^{\varphi}_{t,n}(z)=\Delta^{\varphi}_n(tz/n).
\end{equation}
The next statement is a scaled version of the preceding corollary.
\begin{cor}\label{mainalpha0}
Let  $\varphi \in \Phi$. Then for all $\alpha \in [0,2], n\in \N$ and $t >0,$
\begin{align}
 \|z^{-\alpha}\Delta^{\varphi}_{t,n}\|_{\Wip(\C_+)}\le 8 \left(|\varphi''(0+)|\frac{t}{n}\right)^{\alpha/2},
  \label{mainalpha}
  \end{align}
  and
  \begin{align} \|z^{-\alpha}E^{\varphi}_{t,n}\|_{\Wip(\C_+)}\le 8
  \left(|\varphi''(0+)|\frac{t^2}{n}\right)^{\alpha/2}. \label{mainalpha1}
\end{align}
\end{cor}

\begin{proof}
Fix $n \in \mathbb N$ and $t >0.$
By \eqref{relation},
\begin{equation*}\label{reld}
z^{-\alpha}\Delta^{\varphi}_{t,n}(z)=n^{-\alpha} (z/n)^{-\alpha} \Delta_{nt}(z/n),
\end{equation*}
hence
\begin{eqnarray*}
\|z^{-\alpha}\Delta^{\varphi}_{t,n}\|_{\Wip(\mathbb C_+)}&=&n^{-\alpha}
\|z^{-\alpha} \cdot \Delta_{nt}\|_{\Wip(\mathbb C_+)}\\
&\le& 8 n^{-\alpha} (nt |\varphi''(0+)|)^{\alpha/2}\\
&=& 8 \left(\frac{t}{n}|\varphi''(0+)|\right)^{\alpha/2},
\end{eqnarray*}
and \eqref{mainalpha} is proved.

To prove \eqref{mainalpha1} observe that
\begin{equation*}\label{rele}
z^{-\alpha}E^{\varphi}_{t,n}(z)=\left({t}/{n}\right)^{\alpha} (t z/n)^{-\alpha} \Delta_{n}(t z/n).
\end{equation*}
Thus
\begin{eqnarray*}
\|z^{-\alpha}E^{\varphi}_{t,n}\|_{\Wip(\mathbb C_+)}&=&\left({t}/{n}\right)^{\alpha}
 \|z^{-\alpha} \Delta_{n}(\cdot)\|_{\Wip(\mathbb C_+)}\\
&\le& 8 \left({t}/{n}\right)^{\alpha} (n|\varphi''(0+)|)^{\alpha/2}\\
&=& 8  \left(|\varphi''(0+)|\frac{t^2}{n}\right)^{\alpha/2},
\end{eqnarray*}
and  \eqref{mainalpha1} follows.
\end{proof}

Let us now apply Corollary \ref{mainalpha0} to the approximations
of the exponential function considered in Theorem
\ref{introapprox}.
\begin{exa}\label{Exampp}

$a)$ {\it Yosida's approximation.} Let  $\varphi_1$ be a Bernstein function defined by
\[
\varphi_1(z):=\frac{z}{z+1},\qquad z>0.
\]
Note that $\varphi_1 \in \Phi$ since
\[
\varphi_1(0)=0,\quad \varphi'_1(0+)=1,\quad \varphi''_1(0+)=-2.
\]

Then \eqref{mainalpha} and \eqref{mainalpha1} imply that
\begin{eqnarray}
 \|z^{-\alpha} \Delta^{\varphi_1}_{t,n}\|_{\Wip(\C_+)} \le 16
 \left(\frac{t}{n}\right)^{\alpha/2},\label{y1}
 \end{eqnarray}
 and
 \begin{eqnarray}
 \|z^{-\alpha}  E^{\varphi_1}_{t,n}\|_{\Wip(\C_+)} \le 16
 \left(\frac{t^2}{n}\right)^{\alpha/2},\label{y2}
\end{eqnarray}
for all $t >0$ and $n \in \mathbb N.$

$b)$ \emph{Dunfod-Segal approximation.} Define a  Bernstein function
$\varphi_2$ as
\[
\varphi_2(z):=1-e^{-z},\qquad z>0.
\]
Since
\[
\varphi_2(0)=0,\quad \varphi'_2(0+)=1,\quad \varphi''_2(0+)=-1,
\]
we have $\varphi_2 \in \Phi.$ Now  \eqref{mainalpha} and
\eqref{mainalpha1} yield
\begin{eqnarray*}
 \|z^{-\alpha}  \Delta^{\varphi_2}_{t,n}\|_{\Wip(\C_+)} \le 8
 \left(\frac{t}{n}\right)^{\alpha/2},
 \end{eqnarray*}
 and
 \begin{eqnarray*}
 \|z^{-\alpha}  E^{\varphi_2}_{t,n}\|_{\Wip(\C_+)} \le 8 \left(\frac{t^2}{n}\right)^{\alpha/2},
\end{eqnarray*}
for all $t >0$ and $n \in \mathbb N.$

 $c)$ \emph{Euler's
approximation.} Consider the  Bernstein function
\[
\varphi_{3}(z):=\log (1+z),\quad z>0.
\]
Since
\[
\varphi_{3}(0)=0,\quad \varphi'_{3}(0+)=1,\quad
\varphi''_{3}(0+)=-1,
\]
we have $\varphi_3 \in \Phi.$ By \eqref{mainalpha} and
\eqref{mainalpha1},
\begin{eqnarray*}
\|z^{-\alpha} \Delta^{\varphi_3}_{t,n}\|_{\Wip(\C_+)} \le
8 \left(\frac{t}{n}\right)^{\alpha/2},
\end{eqnarray*}
and
\begin{eqnarray*}
 \|z^{-\alpha} E^{\varphi_3}_{t, n}\|_{\Wip(\C_+)} \le 8
\left(\frac{t^2}{n}\right)^{\alpha/2},
\end{eqnarray*}
for all $t >0$ and $n \in \mathbb N.$

\end{exa}

\section{Applications: rates of approximation of bounded $C_0$-semigroups}\label{applgeneral}
Using  Corollary \ref{mainalpha0} and the HP functional calculus
set up in Section \ref{fcc}, we are now able to get convergence
rates in abstract and concrete approximation schemes for
$C_0$-semigroups on Banach spaces. Without loss of generality, we
can consider only bounded $C_0$-semigroups. (The general case can
be reduced to this case by rescaling.)

 If $-A$ is the generator of a bounded $C_0$-semigroup $(e^{-tA})_{t \ge 0}$
on a Banach space $X,$ then for $t >0$ and $n \in \N$ let
\begin{equation*}
 \Delta^{\varphi}_{t,n}(A):=e^{-nt\varphi(A/n)}-e^{-tA},
 \end{equation*}
 and
 \begin{equation*}
  E^{\varphi}_{t, n}(A):= e^{-n\varphi(t A/n)}-e^{-tA},
\end{equation*}
where the right-hand sides are well-defined by Theorem \ref{subord}.

We need to define the fractional powers of $-A$ in the framework of the extended HP calculus.
Note that
\begin{equation}\label{regular}
z^{\alpha}v_{\alpha,1}(z)=u_{\alpha,1}(z) \quad \text {and} \quad
z^{-\alpha}u_{\alpha,1}(z)=v_{\alpha,1}(z), \quad z \in \mathbb
C_+,\quad \alpha >0,
\end{equation}
and $v_{\alpha, 1}, u_{\alpha,1} \in \Wip (\mathbb C_+)$ by Lemma
\ref{Propbeta}. Thus, by the first equality in \eqref{regular},
the function $z \to z^{\alpha},\alpha >0,$ belongs to the extended
HP calculus for $A$. If $A$ is injective then $A^{\alpha}$ is
injective as well (see \cite[Proposition 3.1.1, d)]{Ha06}), and by
the second equality in \eqref{regular}, the function $z\to
z^{-\alpha}, \alpha >0,$ belongs to the extended HP calculus for
$A$ too. It is  instructive to observe that the
fractional powers defined in this way coincide with the fractional
powers defined by means of the extended holomorphic functional
calculus developed e.g. in \cite[Section 3]{Ha06}. We omit the
details and refer to \cite{BaGoTa}, where, in particular,
compatibility of various calculi is discussed. (See also
\cite[Proposition 1.2.7 and Section 3.3]{Ha06}.)

Now we are in position to recast  estimates of scalar functions
from the previous section into norm estimates of functions of
generators.
\begin{thm}\label{ThC}
Let  $-A$ be the generator of a bounded $C_0$-semigroup $(e^{-tA})_{t \ge 0}$
on a Banach space $X$, and
let $\varphi \in \Phi.$ If $\alpha \in (0,2],$
then for all $x \in \dom (A^\alpha), t >0,$ and $n \in \mathbb N,$
\begin{equation}
\|\Delta^{\varphi}_{t,n}(A)x\|\le 8 M
\left(\frac{t|\varphi''(0+)|}{n}\right)^{\alpha/2}\|A^\alpha
x\|,\label{Estt1} \end{equation} and
\begin{equation}
\|E^{\varphi}_{t,n}(A)x\|\le 8 M
\left(\frac{t^2|\varphi''(0+)|}{n}\right)^{\alpha/2}\|A^\alpha
x\|, \label{Estt1a}
\end{equation}
where $M:=\sup_{t\ge 0}\|e^{-tA}\|$.
\end{thm}
\begin{proof}
Let $n \in \mathbb N, t >0,$ $\alpha \in (0,2]$ and $x \in \dom
(A^\alpha)$ be fixed. We use the HP calculus to estimate the norm
of  $(A+\delta)^{-\alpha} \Delta_{t,n}(A+\delta), \delta >0.$ Since $\dom
(A^\alpha)=\dom(A+\delta)^\alpha$
 for every $\delta >0,$
  \eqref{hille-pillips} and \eqref{mainalpha} imply that
\begin{equation}\label{delta}
\| \Delta^\varphi_{t,n}(A+\delta) x \| \le 8 M\left(\frac{t}{n}|\varphi(0+)|\right)^{\alpha/2}
 \| (A+\delta)^{\alpha}x \|.
\end{equation}
By \cite[Proposition 3.1.9]{Ha06},
\begin{equation}
\lim_{\delta \to 0+}(A+\delta)^\alpha x = A^\alpha x.
\end{equation}
Furthermore, from Theorem \ref{subordf} it follows that
\begin{equation}
 \Delta_{t,n}^{\varphi}(\cdot +\delta)-
 \Delta_{t,n}^{\varphi}(\cdot)=\int_{0}^{\infty}e^{-zs}(e^{-\delta s}-1)\,
 d \nu_{t,n}(s),
\end{equation}
where $\nu_{n,t}$ are  bounded Radon  measures on $\mathbb R_+.$
Hence, by the Lebesgue bounded convergence theorem,
\begin{equation*}
\Wip(\mathbb C_+)-\lim_{\delta \to 0+} \Delta^{\varphi}_{t,n}(\cdot +\delta)
= \Delta^{\varphi}_{t,n},
\end{equation*}
and by \eqref{hille-pillips},
\begin{equation*}
\lim_{\delta \to 0+} \Delta^{\varphi}_{t,n}(A +\delta)x
= \Delta^{\varphi}_{t,n}(A)x.
\end{equation*}
Letting now   $\delta \to 0+$  in \eqref{delta} we get
\begin{equation*}
\| \Delta^{\varphi}_{t,n}(A) x \| \le 8 M\left(\frac{t}{n}|\varphi(0+)|\right)^{\alpha/2} \| A^{\alpha}x \|,
\end{equation*}
and  \eqref{Estt1} is proved.

The proof of \eqref{Estt1a} is completely analogous to the proof of \eqref{Estt1} being based on
 \eqref{hille-pillips} and \eqref{mainalpha1}, and is therefore omitted.
\end{proof}
\begin{rem}
If $\alpha=2$ then \eqref{Estt1} and \eqref{Estt1a} hold with
better constants. Indeed, arguing as in the proof of Theorem
\ref{ThC} and using Theorem \ref{FP}, we obtain
\begin{equation}
\|\Delta^{\varphi}_{t,n}(A)x\|\le
 M \frac{t|\varphi''(0+)|}{2n}\|A^2 x\|,\label{Estt11}
 \end{equation}
 and
 \begin{equation}
\|E^{\varphi}_{t,n}(A)x\|\le  M \frac{t^2|\varphi''(0+)|}{2n}\|A^2
x\|, \label{Estt11a}
\end{equation}
for all $n\in \N, t>0$ and $x \in \dom(A^2).$
\end{rem}

\begin{rem}\label{rangealpha}
Observe that Theorem \ref{ThC} does not in general hold  if  $\alpha>2.$ To see this, we first note that
if $\varphi \in \Phi$ and $t>0$ then
\[
\lim_{z\to 0,\;{\rm Re}\,z\ge 0}\frac{\Delta^{\varphi}_{t,n}(z)}{z^2}=
\lim_{z\to 0,\,{\rm Re}\,z\ge 0}\,\frac{e^{-nt\varphi(z/n)}-e^{-tz}}{z^2}=\frac{t}{2n}|\varphi''(0+)|.
\]
Hence if   $\varphi''(0+)\not = 0$
(that is if $\varphi(z) \not\equiv z$), and if $-A$ is a
generator of a bounded $C_0$-semigroup on a Banach space $X$  such that
$\cls{\ran}(A)=X$ and
$\sigma(A)$ has accumulation point at zero, then
$A^{-\alpha}\Delta_{t,n}^\varphi(A)$ is not bounded for all $t >0$ and $n \in \mathbb N.$
Indeed, if $A^{-\alpha}\Delta_{t,n}^\varphi(A)$ is bounded for some $t>0$ and $n \in \mathbb N,$
then, using the product rule \eqref{hpfc.e.prod}, for any $\tau>0$ we have
\[
\|v_{\alpha,\tau}(A)\Delta_{t,n}^\varphi(A)\|=\|A^{-\alpha}\Delta_{t,n}^\varphi(A)u_{\alpha,\tau}(A)\|\le
Mc_\alpha \|A^{-\alpha}\Delta_{t,n}^\varphi(A)\|.
\]
On the other hand, choosing
 $\{\lambda_k\}_{k=1}^\infty\subset \sigma(A)$ such that $\lambda_k\to 0$ as $k\to\infty$, and $\tau_k:=2^{-1} |\lambda_k|, k \in \mathbb N,$
and employing  Theorem \ref{mapping},
we get
\begin{eqnarray*}
\| v_{\alpha,\tau_k}(A)\Delta^{\varphi}_{t,n}(A)\| &\ge&
\sup_{\lambda \in \sigma(A)}\,|(\lambda+\tau_k)^{-\alpha}
\Delta^{\varphi}_{t,n}(\lambda)|\\
&\ge&
 |(1+ 2^{-1} e^{-i\arg \, \lambda_k})|^{-\alpha}
|\lambda_k^{-\alpha}\Delta^{\varphi}_{t,n}(\lambda_k)|\to \infty,
k \to \infty,
\end{eqnarray*}
a contradiction.

Similarly, $A^{-\alpha} E^{\varphi}_{t,n}(A)$ is an unbounded
operator for $\alpha>2$ if $\varphi''(0+)\not = 0$ ($\varphi(z) \not\equiv z$) and $\sigma(A)$
has accumulation point at zero.
\end{rem}

The following quantification of Theorem \ref{introapprox} is a
direct consequence of Theorem \ref{ThC} and Example \ref{Exampp}.
It is one of the main results of the paper.

\begin{cor}\label{Ex22}
Let  $-A$ be the generator of a bounded $C_0$-semigroup
$(e^{-tA})_{t \ge 0}$ on a Banach space $X,$ and let $\alpha \in (0,2].$ Then for
all $x\in \dom(A^\alpha),\quad t>0,$ and $n \in \mathbb N,$
\begin{itemize}
\item [a)] \, {\rm [Yosida's approximation]}
\begin{equation*}\label{yosid1}
\|e^{-tA}x-e^{-n t A(n+A)^{-1}}x\|
\le 16 M
\left(\frac{t}{n}\right)^{\alpha/2} \|A^\alpha x\|;
\end{equation*}
\item [b)]\, {\rm [Dunford-Segal's approximation]}
\[
\|e^{-tA}x-e^{-nt(1-e^{-A/n})}x\|
\le 8 M \left(\frac{t}{n}\right)^{\alpha/2} \|A^\alpha x\|;
\]

\item [c)]\, {\rm [Euler's approximation]}
\[
\|e^{-tA}x - (1+tA/n)^{-n}x\|\le
8
M\left(\frac{t^2}{n}\right)^{\alpha/2} \|A^\alpha x\|. 
\]
\end{itemize}
where $M:=\sup_{t \ge0} \|e^{-tA}\|.$
\end{cor}

\begin{rem}\label{newapprox}
We may use Theorem \ref{ThC} and Example \ref{Exampp} to get  new
approximation formulas. In particular, applying  \eqref{y2}
instead of \eqref{y1} we obtain the following variation upon
Yosida's approximation. Let $-A$ be the generator of a bounded
$C_0$-semigroup $(e^{-tA})_{t \ge 0}$ on $X,$ and let $\alpha \in
(0,2].$ Then for all $x\in \dom(A^\alpha), t>0,$ and $n \in
\mathbb N,$
\begin{equation}\label{yosid}
\|e^{-tA}x-e^{-n t A(n+tA)^{-1}}x\| \le 8 M
\left(\frac{t^2}{n}\right)^{\alpha/2} \|A^\alpha x\|.
\end{equation}
The estimate \eqref{yosid} is clearly better than Corollary \ref{Ex22},a)
for small $t$, while  Corollary \ref{Ex22},a) gives  approximation better
than \eqref{yosid} for big $t.$ We leave the formulation of other
possible approximation relations to the reader.
\end{rem}
We will prove in Section \ref{optimalsection} below that the
convergence rate estimates provided by Theorem \ref{ThC} (and Corollary
\ref{Ex22}) are sharp under certain spectral conditions on $-A.$

\section{Rates of approximation of analytic semigroups}\label{secanal}

In this section, we improve Theorem \ref{ThC} (and Corollary
\ref{Ex22}) if $(e^{-tA})_{t \ge 0}$ is a bounded analytic
$C_0$-semigroup on a Banach space $X$. Moreover, we get convergence
rates even if the semigroup approximation takes place on the
whole of $X.$

Recall that a $C_0$-semigroup $(e^{-tA})_{t \ge 0}$ on $X$ is said
to be analytic  of angle $\theta$ if it has an analytic
 extension to $S_\theta=\{\lambda \in \mathbb C: |\arg \lambda| <\theta\}$ for
some $\theta \in (0,\frac{\pi}{2}]$ which is bounded on
$S_{\theta'} \cap  \{\lambda \in \mathbb C:   |\lambda| \le 1 \}$
for all $\theta' \in (0,\theta).$ Moreover, a $C_0$-semigroup
$(e^{-tA})_{t \ge 0}$ is called bounded analytic of angle $\theta$
if it has a bounded analytic extension to $S_{\theta'}$ for each
$\theta' \in (0,\theta).$  We call a semigroup $(e^{-tA})_{t \ge
0}$ bounded analytic if it is bounded analytic of angle $\theta$
for some $\theta \in (0,\frac{\pi}{2}].$ For basic properties of
analytic $C_0$-semigroups see e.g. \cite[Chapter 3.7]{ABHN01} and
\cite[Chapter II.4]{EngNag}.

It is well-known that bounded analytic semigroups
can be characterized in terms of their asymptotics on the real
axis. Namely,  $-A$ is the generator of a bounded
analytic semigroup $(e^{-tA})_{t \ge 0}$ on a Banach space $X$ if
and only if $\sup_{t\ge 0}\,\|e^{-tA}\|$ and
$\sup_{t>0}\,\|tAe^{-tA}\|$ are finite, see e.g. \cite[Theorem 4.6]{EngNag}.
 In this case, let for the rest of this paper
\begin{equation}\label{M1}
M_0:=\sup_{t\ge 0}\,\|e^{-tA}\|,\qquad
M_1:=\sup_{t>0}\,\|tAe^{-tA}\|,
\end{equation}
and
\begin{equation}\label{MMG}
 M:=\max \left(2M_0,M_1\right).
\end{equation}
Consider now the Banach space $\mathcal{X}=X\oplus X$ with the
norm $\|\cdot\|_{\mathcal{X}}$ given by
\[
\|(x_1,x_2)\|_{\mathcal{X}}:=\max\left(\|x_1\|,\|x_2\|\right),
\qquad (x_1,x_2) \in \mathcal X.
\]
If $A$ is a closed densely defined operator on $X,$ then define
the operator $\mathcal{A}$ on $\mathcal X$ by
\begin{equation}\label{GGG}
\dom (\mathcal A):= \dom (A)\oplus\dom (A), \quad \mathcal{A}:=
\left(\begin{array}{cc}
A & A\\
0 & A
\end{array}
\right).
\end{equation}
By \cite[p. 367--368]{CrPaTa79} (see also \cite[Chapter 5]{Cast}),
$-\mathcal{A}$ is
 the generator of an  analytic $C_0$-semigroup
 $(e^{-t\mathcal{A}})_{t \ge 0}$ on $\mathcal{X}$
if and only if $-A$ generates an analytic semigroup on $X.$ The
semigroup $(e^{-t\mathcal{A}})_{t \ge 0}$ is given by
\begin{equation}\label{AAA}
e^{-t\mathcal{A}}=
\left(\begin{array}{cc}
e^{-tA} & -tAe^{-tA}\\
0 & e^{-tA}
\end{array}
\right),\quad t> 0.
\end{equation}
Combining the two results cited above, we formulate a
statement which  will be one of the basic tools in this section.
\begin{thm}\label{crandall}
The operator $-A$ is the generator of a  bounded analytic
$C_0$-semigroup  $(e^{-tA})_{t \ge 0}$ on $X$ if and only if
$-{\mathcal A}$ defined by means of \eqref{GGG} is the generator
of a
 bounded analytic $C_0$-semigroup
 $(e^{-t\mathcal{A}})_{t \ge 0}$ on $\mathcal{X}$ given by
 \eqref{AAA}.
Moreover, if $(e^{-tA})_{t \ge 0}$ satisfies \eqref{M1}, then
\begin{eqnarray}
\sup_{t\ge 0}\,\|e^{-t\mathcal{A}}\|&\le& M_0+M_1,\label{tildeM}\\
\sup_{t >0}\, \|t {\mathcal A} e^{-t\mathcal{A}}\|&\le& 2M_1 + 4
M_1^2.\label{tildeM1}
\end{eqnarray}
\end{thm}

If $\varphi$ is a
Bernstein function, then by Theorem \ref{subord}, $\varphi(A)$ is
the generator of a bounded $C_0$-semigroup
$(e^{-t\varphi(\mathcal{A})})_{t\ge 0}$ on $\mathcal X.$ Moreover,
subject to a mild restriction, $(e^{-t\varphi(\mathcal{A})})_{t\ge
0}$ has a matrix structure similar to \eqref{AAA} as the next
proposition shows.  \begin{prop} Let $-A$ be the generator of a
bounded analytic $C_0$-semigroup $(e^{-tA})_{t \ge 0}$ on $X$ and
let $\varphi$ be a Bernstein function such that
\begin{equation}\label{der1}
\varphi'(0+)<\infty.
\end{equation}
Then
$(e^{-t\varphi(\mathcal{A})})_{t\ge 0}$ is a bounded $C_0$-semigroup
on $\mathcal{X}$ and
\begin{equation}\label{BerAn}
e^{-t\varphi(\mathcal{A})}=
\left(\begin{array}{cc}
e^{-t \varphi(A)} & -tA\varphi'(A)e^{-t\varphi(A)}\\
0 & e^{-t\varphi(A)}
\end{array}
\right),\quad t>0.
\end{equation}
\end{prop}

\begin{proof}
By Theorem \ref{subord}, $-\varphi(A)$ is the generator of a
$C_0$-semigroup $(e^{-t\varphi(A)})_{t\ge 0}$ on $X$ and there
exists a convolution semigroup of subprobability measures
$(\nu_t)_{t \ge0}$ on $\mathbb R_+$ such that
\[
e^{-t\varphi(A)}=\int_0^\infty e^{-s A}\,\nu_t(ds)
\]
(where the integral converges strongly). Hence, by Theorem
\ref{crandall},
\begin{eqnarray}\label{matrix}
e^{-t\varphi(\mathcal{A})}&=&
\int_0^\infty
\left(\begin{array}{cc}
e^{-sA} & -sAe^{-sA}\\
0 & e^{-sA}
\end{array}
\right)\,\nu_t(ds)\\
&=&\left(\begin{array}{cc}
e^{-t \varphi(A)} & -\int_0^\infty s A e^{-sA}\,\nu_t(ds)\\
0 & e^{-t\varphi(A)} \notag
\end{array}
\right).
\end{eqnarray}
By \eqref{der1}, we have  $\varphi' \in \Wip(\mathbb C_+).$ Moreover $e^{-t\varphi} \in \Wip(\mathbb C_+)$ as well.
Hence by the product rule for the Hille-Phillips functional calculus we have
\begin{equation*}
-t \varphi'(A) e^{-t\varphi(A)}= [-t
\varphi'e^{-t\varphi}](A)=[e^{-t\varphi}]'(A)=-\int_{0}^{\infty}se^{-sA}
\, \nu_t(ds),
\end{equation*}
and
\[
-\int_0^\infty s A e^{-sA}\,\nu_t(ds) = -A \int_0^\infty  se^{-sA}\, \nu_t(ds)=-tA\varphi'(A) e^{-t\varphi(A)}.
\]
In view of \eqref{matrix}, this yields (\ref{BerAn}).
\end{proof}
\begin{rem}
One can prove that if $\varphi$ is a Bernstein function then  $(e^{-t\varphi(\mathcal{A})})_{t\ge 0}$ is a
bounded analytic $C_0$-semigroup if $(e^{-tA})_{t\ge
0}$ is so. This fact however will not be needed in the sequel. For
its proof as well as other related statements see \cite{GoTo13}.
\end{rem}
The following result  provides convergence rates in approximation formulas for bounded
analytic semigroups on the domains of their generators.
\begin{thm}\label{thtr}
Let  $-A$ be the generator of a  bounded analytic
$C_0$-semigroup $(e^{-tA})_{t\ge 0}$ on $X$, and let $\varphi \in
\Phi.$ Then for all  $t>0$ and $x \in
\dom(A),$
\begin{equation}\label{AnEst1}
\|\Delta^{\varphi}_t(A) x\|\le \left(2M_0+ M_1\right)
|\varphi''(0+)|\|A x\|.
\end{equation}
\end{thm}

\begin{proof}
Without loss of generality we may assume that $A$ is injective.
(Otherwise we can proceed as in the proof of Theorem \ref{ThC}.)

Let  $t>0$ be fixed. We apply  \eqref{Estt11} (with $n=1$) to the
semigroup $(e^{-t\mathcal{A}})_{t\ge 0}$ taking into account the
matrix representations (\ref{AAA}) and (\ref{BerAn}) for
$(e^{-t\mathcal{A}})_{t\ge 0}$ and
$(e^{-t\mathcal{\varphi(A})})_{t\ge 0},$
 respectively. If  $x\in \dom(A^2)$ then $y=(0,x)\in \dom(\mathcal{A}^2)$ and
\begin{eqnarray*}
\left\|tAe^{-tA}x-
t A \varphi'(A)e^{-t\varphi(A)}x
\right\|
&\le&
\|e^{-t \varphi(\mathcal{A})}y-e^{-t\mathcal{A}}y\|_{\mathcal{X}}\\
&\le&  2^{-1} t(M_0+M_1) |\varphi''(0+)|\|\mathcal{A}^2y\|_{\mathcal{X}}\\
&\le&  t (M_0+M_1)\,| \varphi''(0+)|\|A^2x\|.
\end{eqnarray*}
Hence for every $x \in \dom (A),$
\begin{equation}\label{Pred1}
\left\|e^{-tA}x-\varphi'(A)e^{-t\varphi(A)}x \right\| \le
(M_0+M_1)|\varphi''(0+)| \|Ax\|.
\end{equation}
Since $\varphi \in \Phi,$
\begin{equation*}
\varphi'(z)=c+ \int_{0+}^{\infty} e^{-zs} \,s \mu(ds), \qquad z
>0,
\end{equation*}
for some $c >0$ and a bounded positive Radon measure $s\cdot \mu$
on $(0,\infty)$ such that
$$\int_{0+}^{\infty}\,s\mu(ds)=1, \qquad \int_{0+}^{\infty} s^2 \, \mu(ds) < \infty.$$
Therefore, by the HP calculus,
\begin{equation*}
(1-\varphi'(A))x=(\varphi'(0+)-\varphi'(A))x=\int_{0+}^\infty
(1-e^{-sA}) x \,s \mu(ds),
\end{equation*}
and then, in view of Theorem \ref{subord},
\begin{equation*}
e^{-t\varphi(A)}(1-\varphi'(A))x =
\int_{0}^{\infty}\int_{0+}^{\infty} (e^{-\tau A}-e^{-(s+\tau)A})x
\,  \, \mu(ds) \nu_t(d\tau)
\end{equation*}
for a  subprobability Radon measure $\nu_t.$ Thus, if $x \in \dom
(A),$ then
\begin{align}\label{Diff1}
 \|e^{-t\varphi(A)}(\varphi'(A)-1)x\| \le&
  \int_{0}^{\infty} \int_{0+}^\infty \| (e^{-\tau A}-e^{-(s+\tau)A}) x \|\,s \mu(ds)\nu_t(d\tau) \\
 =&  \int_{0}^{\infty}
 \int_{0+}^\infty  \int_{\tau}^{s+\tau} \|e^{-r A}Ax \|\,d r \,s\mu(ds) \, d\nu_t(d\tau) \notag\\
 \le& M_0 \|Ax\| \int_{0+}^{\infty} s^2 \, \mu (ds)\notag\\
  =& M_0 \|Ax\| |\varphi''(0+)|.\notag
\end{align}

Combining (\ref{Pred1}) and (\ref{Diff1}), we obtain
\begin{equation*}\label{eet1XXDD}
\left\|e^{-tA}x-e^{-t\varphi(A)}x \right\| \le \left(2M_0+ M_1
\right)\,|\varphi''(0)| \|Ax\|,
\end{equation*}
for $x\in \dom(A),$ which is (\ref{AnEst1}).

\end{proof}

Our aim now is to prove an estimate of the form
\begin{equation}\label{estimatee}
\|e^{-tA} -e^{- t\varphi(A)}\|\le \frac{C}{t}, \qquad \quad t >0.
\end{equation}
This will be done by writing, as in the proof of Theorem \ref{thtr},
\begin{equation}\label{sum}
\|e^{-tA} -e^{- t\varphi(A)}\| \le \|e^{-tA} - \varphi'(A) e^{-
t\varphi(A)}\| +\|(1- \varphi'(A)) e^{- t\varphi(A)}\|
\end{equation}
and estimating each term in the right-hand side of \eqref{sum}
separately. The first term can be handled easily by our techniques.

\begin{lemma}\label{leminterm}
Let  $-A$ be the generator of a  bounded analytic
$C_0$-semigroup $(e^{-tA})_{t\ge 0}$ on $X$ and let  $\varphi \in
\Phi.$ Then
\begin{equation}\label{AnEst1a}
\|e^{-tA} - \varphi'(A) e^{- t\varphi(A)}\|\le \frac{2(M_0+2
M_1)+4M_1^2}{t}|\varphi''(0+)|.
\end{equation}

\end{lemma}
\begin{proof}
As in the proof of Theorem \ref{thtr} we can assume that $A$ is injective.
 If $y=(0,x)\in \dom(\mathcal{A})$, $x\in \dom(A)$, then  applying  (\ref{AnEst1}) to the
bounded analytic semigroup $(e^{-t\mathcal{A}})_{t
\ge 0}$ we get
\begin{eqnarray*}
\left\|tAe^{-tA}x- tA \varphi'(A)
e^{- t\varphi(A)}x
\right\| &\le&  \left\|e^{-t \mathcal A} y-e^{- t\varphi(\mathcal A)}y
\right\|_{\mathcal X}\\
 &\le& \left(2(M_0+2M_1) +4M_1^2\right)|\varphi''(0+)| \|Ax\|,
\end{eqnarray*}
so
\begin{equation*}\label{latter}
\left\|e^{-tA}x- \varphi'(A) e^{- t\varphi(A)}x \right\| \le
\frac{2(M_0+2M_1)+4M_1^2}{t}|\varphi''(0+)| \|x\|.
\end{equation*}
Since $\dom(A)$ is dense in $X,$ the latter inequality holds for
every $x \in X$, and \eqref{AnEst1a} follows.
\end{proof}

To estimate the second term in the right-hand side of \eqref{sum} we need to prove
several auxiliary inequalities.

\begin{lemma}\label{P1G}
Let $p\not\equiv\mbox{const}$ be a Bernstein function and let
\[
q\in L^\infty_{{\rm loc}}(\mathbb R_+) \cap L^1(\R_{+}).
\]
Then
\[
\int_0^\infty e^{-tp(s)}|q(s)|\,ds\le Ct^{-1},\qquad t>0,
\]
where
\[
C=\inf_{a>0}\, \left[\frac{\|q\|_{L^\infty[0,a]}}{p'(a)}+
\frac{\|q\|_{L^1[a,\infty)}}{p(a)}\right].
\]
\end{lemma}

\begin{proof}
Since $p$ is Bernstein, $p$ is increasing on $[0,\infty)$, while
$p'$ is decreasing on $(0,\infty).$  Hence, for all  $a>\epsilon
>0,$
\[
p(s)=\int_\epsilon^s p'(t)\,dt+p(\epsilon)\ge p'(a)
(s-\epsilon),\qquad s\in [0,a],
\]
and, letting $\epsilon \to 0+,$ we get
\[
p(s) \ge p'(a) s,\qquad s\in [0,a].
\]
Therefore, taking into account that  $p'$ is strictly positive on
$(0,\infty)$ by assumption,
\begin{eqnarray*}
\int_0^\infty e^{-tp(s)}|q(s)|\,ds &\le& \int_0^a
e^{-tp(s)}|q(s)|\,ds +\int_a^\infty
e^{-tp(s)}|q(s)|\,ds \\
&\le& \|q\|_{L^\infty[0,a]}\int_0^a
e^{-tp'(a)s}\,ds +e^{-tp(a)}\int_a^\infty |q(s)|\,ds\\
&\le& \left[\frac{\|q\|_{L^\infty[0,a]}}{p'(a)}+
\frac{\|q\|_{L^1[a,\infty)}}{p(a)}\right]t^{-1},\qquad a>0,
\end{eqnarray*}
and the statement follows.
\end{proof}

\begin{lemma}\label{lemmaxG}
Let $-A$ be the generator of a  bounded analytic $C_0$-semigroup
$(e^{-tA})_{t \ge 0}$ on a Banach space $X.$
Then
\begin{equation}\label{maxG}
\|(1-e^{-s A})\,e^{-\tau A}\|\le \frac{2M s}{\tau+s},
\end{equation}
where $M$ is given by \eqref{MMG}.
\end{lemma}
\begin{proof}
By assumption for all  $s,\tau>0,$
\begin{eqnarray*}
\|(1-e^{-s A})\,e^{-\tau A}\|\le 2M_0, 
\end{eqnarray*}
and
\begin{eqnarray*}
\|(1-e^{-s A})\,e^{- \tau A}\|\le \int_{\tau}^{s+\tau}
\|Ae^{-rA}\|\,dr \le M_1 \int_{\tau}^{s+\tau}\, \frac{dr}{r}\le
M_1\frac{s}{\tau}.
\end{eqnarray*}
Then, using \eqref{elementary} with $\beta =1$, we obtain
\begin{equation*}
\|(1-e^{-s A})\,e^{-\tau A}\|\le \min\,(2M_0,M_1s\tau^{-1})\le
\frac{2M s}{\tau+s}.
\end{equation*}
\end{proof}

Using Lemma \ref{lemmaxG} and  Proposition \ref{P1G}, we prove now
a statement which can be considered as a generalization of the
second inequality in \eqref{tildeM1}.
\begin{thm}\label{AunboundedG}
Let $\psi$ be a bounded Bernstein  function satisfying
\begin{equation}\label{CGG}
\psi(0)=0,\qquad \psi'(0+)<\infty,
\end{equation}
and let $\varphi\not\equiv \mbox{const}$ be a Bernstein function.
Let $-A$ be the generator of a  bounded analytic $C_0$-semigroup
$(e^{-tA})_{t \ge 0}$ on $X$.  Then
\begin{equation}\label{ABC1G}
\sup_{t>0}\, \|t \psi(A)e^{-t\varphi(A)}\|\le
2M\left[\frac{\psi'(0+)}{\varphi'(1)}+\frac{\psi(1)}{\varphi(1)}\right].
\end{equation}
\end{thm}

\begin{proof}
Let us start with establishing appropriate integral
representations for $\psi(A)$ and $e^{-t\varphi(A)}, t\ge 0.$

 Since $\psi$ is bounded on $\mathbb R_+$
and $\psi(0)=0,$ the Levy-Hintchine representation of $\psi$ is of
the form
\begin{equation*}\label{derG}
\psi(z)=\int_{0+}^\infty (1 -e^{-zs}) \,\gamma(ds),\quad z>0,
\end{equation*}
where $\gamma$ is a bounded positive Radon measure on
$(0,\infty).$ Then by the HP calculus,
\begin{equation}\label{berG}
\psi(A)=\int_{0+}^\infty (1 -e^{-sA}) \,\gamma(ds),\quad z>0.
\end{equation}

Moreover, by Theorem \ref{subordf},
\begin{equation}\label{AGG}
e^{-t\varphi(z)}=\int_0^\infty e^{- z\tau}\,\nu_t(d\tau),\quad
z>0,
\end{equation}
where $(\nu_t)_{t \ge 0}$ is  a convolution  semigroup of
subprobability Radon measures on $\mathbb R_+.$ From  Theorem
\ref{subord} it follows that $-\varphi (A)$ is the generator of a
bounded $C_0$-semigroup $(e^{-t\varphi(A)})_{t \ge 0}$ given by
\begin{equation}\label{AG}
e^{-t\varphi(A)}=\int_0^\infty e^{- \tau A}\,\nu_t(d\tau),\quad
z>0.
\end{equation}
 Fix $t>0.$ Then
\eqref{berG}, \eqref{AG} and Fubini's theorem imply that
\begin{eqnarray*}
\psi(A)e^{-t\varphi(A)}&=& \int_{0+}^\infty (1-e^{-s
A})\,\gamma(ds)
\int_0^\infty e^{-\tau A}\,\nu_t(d\tau)\\
& =&\int_{0+}^\infty \int_0^\infty (1-e^{-s A})\, e^{-\tau
A}\,\nu_t(d\tau)\,\gamma(ds).
\end{eqnarray*}
Hence, using Lemma \ref{lemmaxG},  we get
\begin{equation}\label{mainG}
\|\psi(A)e^{-t\varphi(A)}\| \le 2M\int_{0+}^\infty s \int_0^\infty
\,\frac{\nu_t(d\tau)}{s+\tau}\,\gamma(ds).
\end{equation}
On the other hand, by (\ref{AGG}), we have
\[
\int_0^\infty \frac{\nu_t(d\tau)}{s+\tau}= \int_0^\infty
e^{-zs}e^{-t\varphi(z)}\,dz,\quad t>0,\quad s>0,
\]
so, in view of (\ref{berG}),
\begin{eqnarray*}
\int_{0+}^\infty s \int_0^\infty
\,\frac{\nu_t(d\tau)}{s+\tau}\,\gamma(ds)&=&\int_{0+}^\infty s
\int_0^\infty e^{-zs}e^{-t\varphi(z)}\,dz \,\gamma(ds)\\
&=&\int_0^\infty e^{-t\varphi(z)} \int_{0+}^\infty s e^{-zs}\,
\,\gamma(ds)\,dz\\
&=&\int_0^\infty e^{-t\varphi(z)} \psi'(z)\,dz.
\end{eqnarray*}
From here by (\ref{mainG}) it follows that
\begin{equation}\label{mainG1}
\|\psi(A) e^{-t\varphi(A)}\| \le 2M \int_0^\infty
e^{-t\varphi(z)}\psi'(z)\,dz,
\end{equation}
where, by Lemma \ref{boundedbe}, $\psi'  \in L^\infty_{{\rm
loc}}(\mathbb R_+) \cap L^1(\R_{+}).$

 Now  \eqref{mainG1} and Lemma
\ref{P1G} with $p=\varphi$ and $q=\psi'$ yield
\[
t\|\psi(A)e^{-t\varphi(A)}\|\le 2M
\inf_{a>0}\,\left\{\frac{\|\psi'\|_{L^\infty[0,a]}}{\varphi'(a)}+
\frac{\|\psi'\|_{L^1[a,\infty)}}{\varphi(a)}\right\},\qquad t>0.
\]
It remains to note that
\begin{eqnarray*}
\inf_{a>0}\,\left\{\frac{\|\psi'\|_{L^\infty[0,a]}}{\varphi'(a)}+
\frac{\|\psi'\|_{L^1[a,\infty)}}{\varphi(a)}\right\}&=& \inf_{a>0}\,
\left\{\frac{\psi'(0+)}{\varphi'(a)}+
\frac{\psi(a)}{\varphi(a)}\right\} \\
&\le& \frac{\psi'(0+)}{\varphi'(1)}+\frac{\psi(1)}{\varphi(1)}.
\end{eqnarray*}
\end{proof}

In particular, Theorem \ref{AunboundedG} provides  a desired estimate of the second term
in  the right-hand of \eqref{sum}.
\begin{cor}\label{C11G}
Let $\varphi$ be a Bernstein function such that
\begin{equation}\label{varphi}
\varphi'(0+)=1,\quad |\varphi''(0+)|<\infty.
\end{equation}
Let $-A$ be the generator of a  bounded analytic
$C_0$-semigroup $(e^{-tA})_{t \ge 0}$ on $X$. Then
\begin{equation*}\label{ABC1G11}
\|(1-\varphi'(A))e^{-t\varphi(A)}\|\le
\frac{2M}{t}\left[\frac{|\varphi''(0)|}{\varphi'(1)}+
\frac{\varphi'(1)}{\varphi(1)}\right],
\end{equation*}
for all $t >0$.
\end{cor}

\begin{proof}
Note that by \eqref{varphi}, the function $\psi(z)=1-\varphi'(z),
z>0,$ is Bernstein and satisfies \eqref{CGG}. Moreover, $\psi$ is
bounded since $\varphi'$ is bounded by \eqref{varphi}. Hence to
get \eqref{ABC1G11} it suffices to apply Theorem \ref{AunboundedG}
to a Bernstein function $\varphi$
 and a bounded Bernstein function $\psi$.
\end{proof}

As an immediate consequence of \eqref{sum}, Lemma \ref{leminterm} and Corollary
\ref{C11G} we obtain \eqref{estimatee} which is one of the main
results of the paper.
\begin{thm}\label{mainanalytic1}
Let $-A$ be the generator of a  bounded analytic
$C_0$-semigroup $(e^{-tA})_{t \ge 0}$ on $X$ and let $\varphi \in
\Phi.$ Then
\begin{equation}\label{wholespace}
\|e^{-tA}x -e^{- t\varphi(A)}x\|\le \frac{C}{t}, \qquad
x\in X,
\end{equation}
for all $t >0$, where
\begin{equation*}
C=2(M_0+2M_1+ 2M_1^2)|\varphi''(0+)|+
2M\left[\frac{|\varphi''(0+)|}{\varphi'(1)}+
\frac{\varphi'(1)}{\varphi(1)}\right].
\end{equation*}
\end{thm}
\begin{rem}
It is interesting to note that Theorem \ref{mainanalytic1} provides an estimate for
the difference of semigroups $(e^{-tA})_{t \ge 0}$
and $(e^{- t\varphi(A)})_{t \ge 0}$ at infinity. If $\varphi$ is bounded,
then $\varphi(A)$ is a bounded operator,
and $(e^{-t A})_{t \ge 0}$ is asymptotically close to the semigroup
$(e^{- t\varphi(A)})_{t \ge 0}$ with bounded generator.
\end{rem}

We continue with ``interpolating''  between \eqref{AnEst1} and
\eqref{wholespace}. To this aim we prove an operator analogue of
Theorem \ref{FP}, which is of independent interest.

\begin{thm}\label{FPAn}
Let $-A$ be the generator of a bounded $C_0$-semigroup
$(e^{-tA})_{t \ge 0}$ on $X,$ with dense range. Let $B$ and
$A^{-1}B$ be bounded linear operators on $X$, and denote
\begin{equation*}\label{conditionAn}
\|B\|=a \qquad \text{and} \qquad  \|A^{-1} B\|=b.
\end{equation*}
Then for any Stieltjes function $f,$ the operator $f(A)B$ is
bounded and
\begin{equation}\label{resAn}
\|f(A)B\|\le 2(1+M_0)a f(a/b),
\end{equation}
where $M_0 :=\sup_{t \ge 0}\|e^{-tA}\|.$
\end{thm}

\begin{proof}
Let a Stieltjes $f_0$ be of the form
\begin{equation}\label{f0}
f_0(z)=\int_{0+}^\infty \frac{\nu(d\tau)}{z+\tau},
\end{equation}
where $\nu$ is a Radon measure on $(0,\infty)$ satisfying
\begin{equation*}
\int_{0+}^\infty \frac{\nu(d\tau)}{1+\tau}<\infty.
\end{equation*}
 First
note that by Proposition \ref{resolventst},
\[
f_0(A)x=\int_{0+}^\infty (\tau+A)^{-1}x\,\nu(d\tau)
\]
for every $x\in \ran(A)\subset \dom(f(A)).$ Since $A^{-1}B$ is
bounded we have $\ran(B)\subset \dom(A^{-1})=\ran(A)$ and
\begin{equation}\label{repdeltaAn}
f_0(A)Bx=\int_{0+}^\infty (\tau+A)^{-1} B x\,\nu(d\tau),\quad x\in
X.
\end{equation}
Now we estimate the operator kernel $(\tau+A)^{-1}B$ in the
integral transform above. Observe that
\[
\|(\tau+A)^{-1} B\|\le \frac{M_0a}{\tau},\quad \tau>0.
\]
On the other hand,
\[
\|(\tau+A)^{-1}B\|  \le \|A(\tau+A)^{-1}\|\|A^{-1}B\|\le (1+M_0)
b.
\]
Then, using \eqref{elementary} with $\beta=1$, we obtain
\begin{equation}\label{12aAn}
\|(\tau+A)^{-1} B\|\le \min\, (M_0a \tau^{-1}, (1+M_0) b)
 \le \frac{2(1+M_0)
a}{a/b+\tau}.
\end{equation}

Now \eqref{repdeltaAn} and \eqref{12aAn} imply
\begin{equation}\label{f00}
\|f_0(A)B\|\le \int_{0+}^\infty \|(\tau+A)^{-1}B\|\,\nu(d\tau)\le
\, 2(1+M_0)a f_0(a/b).
\end{equation}
Finally, if  $f$ is Stieltjes then  $f(z)=c_0z^{-1}+c_1+f_0(z),$ where  $c_0,c_1 \ge 0$ and $f_0$ is as in \eqref{f0}.
Hence, by \eqref{f00},
\begin{eqnarray*}
\|f(A)B\| &\le& c_0b + c_1 a + 2(1+M_0) a f_0(a/b)\\
&=&  2a(1+M_0)  \left(\frac{c_0}{2(1+M_0)}b/a+ \frac{c_1}{2(1+M_0)}+f_0(a/b)\right)\\
&\le& 2a(1+M_0) f(a/b).
\end{eqnarray*}
\end{proof}

\begin{rem}
Recall that a closed densely defined linear operator $A$ on $X$ is called sectorial if  $(-\infty,0)\subset \rho (A)$ and
\[
K:=\sup_{\lambda>0}\, \lambda \|(\lambda+A)^{-1}\|<\infty.
\]
As, by \cite[Theorem 2.5,{\it (ii)}]{GoTo}, Proposition \ref{resolventst} holds for sectorial operators $A$
with dense range, the inequality  \eqref{resAn} is true for such operators too if
 $M_0$ is replaced by
$K.$ The proof remains the same.
\end{rem}

Since the functions $z^{-\alpha}, \alpha \in [0,1],$ are
Stieltjes, the next corollary is a direct consequence of Theorem
\ref{FPAn} and is similar to Corollary \ref{coralpha}.
\begin{cor}\label{coralphaAn}
Under the assumptions of Theorem \ref{FPAn}, if $\alpha\in [0,1],$
then
\begin{equation}\label{res1An}
\|A^{-\alpha} B\|\le  2(1+M_0) a^{1-\alpha} b^\alpha.
\end{equation}
\end{cor}

With a special choice of $B,$ Corollary \ref{coralphaAn} becomes
the next approximation estimate.

\begin{thm}\label{IntAn}
Let  $-A$ be the generator of a  bounded analytic
$C_0$-semigroup $(e^{-tA})_{t\ge 0}$ on $X,$ with dense range. If
$\varphi \in \Phi$ and $\alpha \in [0,1]$ then there exists
$C>0$ such that
\begin{equation}\label{EstS1exp1An}
\|A^{-\alpha} \Delta^{\varphi}_{t}(A)\|\le
\frac{C}{t^{1-\alpha}}
\end{equation}
for all $t>0$.
\end{thm}

\begin{proof}
For a fixed $t>0$ set
 $B=
\Delta^{\varphi}_{t}(A)$. By Theorems \ref{thtr} and
\ref{mainanalytic1}  there exist $C_1\ge C_0 >0$ depending on
$M_0,M_1$ and $\varphi''(0+)$ such that (\ref{EstS1exp1An}) holds
for $\alpha=0$ and $\alpha=1$ with $C_0$ and $C_1$ respectively.
By applying  Corollary \ref{coralphaAn} with
\[
a=\frac{C_0}{t},\qquad b= C_1,\qquad t>0,
\]
and with $B$ as above,  we obtain
\[
\|A^{-\alpha} \Delta^{\varphi}_{t}(A)\|\le
 2(1+M_0) C_0^{1-\alpha} C_1^\alpha t^{\alpha-1}, \qquad \alpha \in [0,1].
\]
Then \eqref{EstS1exp1An} follows with
$$C= 2(1+M_0)\max_{\alpha \in [0,1]}C_0^{1-\alpha} C_1^\alpha=2(1+M_0)C_1.$$
\end{proof}

The following scaled version of Theorem \ref{IntAn} yields
(optimal) convergence rates for approximations of bounded analytic
semigroups. It complements and extends Theorem \ref{ThC} in the
case when a $C_0$-semigroup $(e^{-tA})_{t \ge 0}$ is bounded
analytic.
\begin{thm}\label{IntAna}
Let $\varphi \in \Phi,$ and let $\Delta^{\varphi}_{t,n}$ and
$E^{\varphi}_{t,n}$ be defined by \eqref{deltadef} and
\eqref{edef}, respectively. Let $-A$ be the generator of a bounded
analytic $C_0$-semigroup $(e^{-tA})_{t \ge 0}$ on $X$. Then there
exists $C>0$ such that for every $\alpha \in [0,1]$,
\begin{eqnarray}
\|\Delta^{\varphi}_{t,n} (A)x \| &\le& \frac{C}{n t^{1-\alpha}} \|A^{\alpha} x\|, \label{delanap}\\
\| E^{\varphi}_{t,n} (A)x \| &\le& \frac{C t^{\alpha}}{n} \|
A^{\alpha} x\|,  \label{eanap}
\end{eqnarray}
for all $x \in \dom(A^\alpha),$ $n \in \mathbb N$ and $t >0.$
\end{thm}
\begin{proof}
Let $\delta >0, \alpha \in [0,1],$ $n \in \N,$ and $x \in \dom
(A^\alpha)$ be fixed. Then by \eqref{relation} and Theorem
\ref{IntAn}, there exists $C>0$ such that
\begin{equation*}
\|\Delta^{\varphi}_{t,n} (A+\delta)x\| = \|\Delta^{\varphi}_{nt}
((A+\delta)/n)x\| \le \frac{C}{n t^{1-\alpha}}
\|(A+\delta)^{\alpha} x\|.
\end{equation*}
Arguing as in the proof of Theorem \ref{ThC} and letting $\delta
\to 0+$ we get
\begin{equation*}
\|\Delta^{\varphi}_{t,n} (A)x\|  \le \frac{C}{n t^{1-\alpha}} \|A^{\alpha} x\|.
\end{equation*}

Similarly, since
\begin{eqnarray*}
\|E^{\varphi}_{t,n} (A+\delta)x\| &=& \|\Delta^{\varphi}_{n} (t(A+\delta)/n)x\| \\
&\le& \frac{C}{n^{1-\alpha}} \left(\frac{t}{n}\right)^\alpha\|(A+\delta)^{\alpha} x\|\\
&=& \frac{C t^{\alpha}}{n} \|(A+\delta)^{\alpha} x\|
\end{eqnarray*}
by  \eqref{relation} and Theorem \ref{IntAn},
we obtain as above
\begin{eqnarray*}
\|E^{\varphi}_{t,n} (A)x\| \le
\frac{C t^{\alpha}}{n} \| A^{\alpha} x\|.
\end{eqnarray*}
\end{proof}

We now specify Theorem \ref{IntAna} for the three classical
semigroup approximations considered in Section \ref{applgeneral}.
The next result is an improvement of Corollary \ref{Ex22} for
bounded analytic semigroups.
\begin{cor}\label{Ex2}
Let $-A$ be the generator of a  bounded analytic $C_0$-semigroup
$(e^{-tA})_{t \ge 0}$ on $X$. Then there exists $C>0$ such that for every
$\alpha \in [0,1],$
\begin{itemize}
\item [a)] \, {\rm [Yosida approximation]}
\[
\|e^{-tA}x-e^{-n t A(n+A)^{-1}}x\| \le C
(nt^{1-\alpha})^{-1} \|A^\alpha x\|;
\]
\item [b)]\, {\rm [Dunford-Segal approximation]}
\[
\|e^{-tA}x-e^{-nt(1-e^{-A/n})}x\| \le
C(nt^{1-\alpha})^{-1} \|A^\alpha x\|;
\]

\item [c)]\, {\rm [Euler approximation]}
\[
\|e^{-tA}x - (1+t/nA)^{-n}x\|\le
 C(nt^{-\alpha})^{-1} \|A^\alpha x\|. 
\]
\end{itemize}
for all
$t
>0, n \in \mathbb N$ and $x \in \dom (A^{\alpha}).$
\end{cor}

\section{Optimality of rate estimates}\label{optimalsection}

In this section, we show that Theorems \ref{ThC}  and \ref{IntAn} as well as their
respective Corollaries \ref{Ex22} and  \ref{Ex2} are sharp under
natural spectral conditions on the generator.

We start with two auxiliary technical lemmas.

\begin{lemma}\label{exp}
If  $\varphi \in \Phi, \varphi(z)\not \equiv z,$ then there exist
$c>0$ and $T >0$ such that
\begin{equation}\label{deltaexp}
|\Delta^{\varphi}_{t}(\pm i/\sqrt{t})|\ge c,
\end{equation}
for all $t \ge T$.
\end{lemma}

\begin{proof}
By assumption, $\varphi$ is of the form
\[
\varphi(z)=bz+\int_{0+}^\infty (1-e^{-zs})\,\mu(ds),
\]
where a positive Radon measure $\mu$ satisfies $\int_{0+}^\infty
\frac{s\,\mu(ds)}{1+s}<\infty.$ Since $\mu \not\equiv 0$ there
exists $r>0$ such that $\mu ((0,r)) >0$ and then $d:=\int_{0+}^{r}
s^2\,\mu(ds) >0.$ Note that for every $\tau \in
(0,\frac{\pi}{T}],$
\[
{\rm Re}\,(\varphi(i\tau)) = 2\int_{0+}^\infty \sin^2(\tau
s/2)\,\mu(ds)\ge \frac{2 \tau^2 }{\pi^2}\int_{0+}^{\pi/\tau}
s^2\,\mu(ds)\ge \frac{2 \tau^2 d}{\pi^2}.
\]
Hence for any $\tau\in (0,\frac{\pi}{T}],$
\[
|\Delta^{\varphi}_{t}(-i\tau)|=|\Delta^{\varphi}_{t}(i\tau)|\ge
1-e^{-t{\rm Re}\,\varphi(i\tau)}\ge 1-e^{-2d t\tau^2/\pi^2}.
\]
Setting $\tau =1/\sqrt{t}$, we thus obtain
\[
|\Delta^{\varphi}_t(-i/\sqrt{t})|=|\Delta^{\varphi}_t(i/\sqrt{t})|\ge
1-e^{-2d/\pi^2},\quad t\ge \frac{r^2}{\pi^2}.
\]
It remains to put $c:=1-e^{-2d/\pi^2}>0$ and
$T=\frac{r^2}{\pi^2}.$
\end{proof}

\begin{lemma}\label{lemmaL}
 If  $\varphi \in \Phi, \varphi(z)\not \equiv z,$ then there exist $\delta>0$ and $c>0$
 such that
 \begin{equation}\label{AH}
e^{-\varphi(\tau)/\tau}-e^{-1}\ge c \tau, \qquad\tau \in (0,\delta].
 \end{equation}
\end{lemma}
\begin{proof} By applying Lopital's rule twice, we obtain
\begin{eqnarray*}
\lim_{\tau\to
0+}\,\frac{e^{-\varphi(\tau)/\tau}-e^{-1}}{\tau}&=&\lim_{\tau\to
0+}\frac{d}{d\tau}e^{-\varphi(\tau)/\tau}\\
 &=&e^{-1}\lim_{\tau\to
0+}\,\frac{\varphi(\tau)-\tau \varphi'(\tau)}{\tau^2}
=\frac{1}{2e}|\varphi''(0+)|>0.
\end{eqnarray*}
This clearly implies the claim.
\end{proof}

We also need a simple result on  functional calculi which is
convenient to separate as the next lemma.
\begin{lemma}Let  $-A$ be the generator of a bounded $C_0$-semigroup
 on $X$ such that
$\overline{{\rm ran}}\, (A)=X.$ If $\alpha \ge 0$ and
$f(z)=z^{-\alpha}\Delta_t(z), z \in \mathbb C_+,$ then $f$ belongs
to the extended Hille-Philips functional calculus for $A$.
\end{lemma}
\begin{proof}
If $e(z):=\frac{z^{\alpha}}{(z+1)^{\alpha}},$ then $e \in \Wip
(\mathbb C_+)$ by Lemma \ref{betaW}. Moreover $ef \in \Wip
(\mathbb C_+),$ since $(z+1)^{-\alpha} \in \Wip (\mathbb C_+)$ by
Lemma \ref{betaW} again. Thus, $f$  belongs to the extended
Hille-Philips functional calculus if $e(A)$ is injective. To prove
the injectivity  of  $e(A)$ it suffices to note that, since $A$ is
injective, $A^{\alpha}$ is injective as well by \cite[Proposition
3.1.1, d)]{Ha06}.
\end{proof}

Now we are ready to prove the main result of this section.
\begin{thm}\label{sharpHY1}
Let  $-A$ be the generator of a bounded $C_0$-semigroup
$(e^{-tA})_{t\ge 0}$ on a Banach space $X$ such that
$\overline{{\rm ran}}\, (A)=X$, and let  $\varphi \in \Phi,
\varphi(z)\not \equiv z.$
\begin{itemize}
\item [(i)]
If $ \{|s|:\,s\in \mathbb R,\,is\in \sigma(A)\}=\mathbb R_{+}, $
then  there exist $c>0$ and $T>0$ independent of $A$ such that for
every $\alpha\in [0,2]$ and all $t \ge T,$
\begin{equation}\label{EstS1exp}
\|A^{-\alpha}\Delta^\varphi_t(A)\|\ge
c {t}^{\alpha/2}.
\end{equation}

\item [(ii)] If
$\mathbb R_{+}\subset \sigma(A), $ then  there exist $c>0$ and
$T>0$ independent of $A$ such that for every
$\alpha\in [0,2]$ and
all $t\ge T,$
\begin{equation}\label{EstS1exp1}
\|A^{-\alpha}\Delta^\varphi_t(A)\|\ge c t^{\alpha-1}.
\end{equation}
\end{itemize}
\end{thm}

\begin{proof}
Let $\alpha \in (0,2]$ be fixed. Recall that by Corollary
\ref{coralpha}, $z^{-\alpha} \Delta^{\varphi}_{t} \in \Wip(\mathbb
C_+)$ for all $t>0.$ By the product rule \eqref{hpfc.e.prod},
\begin{equation*}
A^{-\alpha} (e^{-t\varphi(A)} - e^{-tA})=(z^{-\alpha}\Delta_t)(A).
\end{equation*}
Hence, by Theorem \ref{mapping},
\begin{equation}\label{InclH1}
\|A^{-\alpha} (e^{-t\varphi(A)} - e^{-tA})\| \ge \sup_{\lambda \in
\sigma(A)}\,|\lambda^{-\alpha} (e^{-t\varphi(\lambda)} -
e^{-t\lambda})|,
\end{equation}
for every $t>0.$
Let $c$ and $T$  be given by Lemma \ref{exp}, and let $t$ be such
that $t \ge T.$ By our assumption either $i/\sqrt{t}\in\sigma(A)$
or $-i/\sqrt{t}\in \sigma(A)$. If $i/\sqrt{t} \in\sigma(A)$ then
from (\ref{InclH1}) it follows that
\begin{equation}\label{ShAexp}
\|A^{-\alpha}(e^{-t\varphi(A)}-e^{-tA}) \| \ge t^{\alpha/2}
|\Delta^{\varphi}_{t}(i/\sqrt{t})|\ge ct^{\alpha/2}.
\end{equation}
Then, by Lemma \ref{exp}, \eqref{ShAexp} implies \eqref{EstS1exp}.

The case  $\lambda=-i/\sqrt{t} \in\sigma(A)$, $t \ge T$, is
considered similarly.

To prove (ii) we argue as in the proof of (i) but use Lemma \ref{exp} instead
of Lemma \ref{lemmaL}. Let $c,\delta >0$ be given by Lemma
\ref{lemmaL} and let $t>0$  be such that $t \ge T=1/\delta$. By
assumption $\lambda=1 /t\in\sigma(A),$ hence from (\ref{InclH1}),
 (\ref{AH}) and the product rule \eqref{hpfc.e.prod} it follows that
\[
\|A^{-\alpha} \Delta^\varphi_{t}(A)\| \ge  t^{\alpha}
|(e^{-t\varphi(1/t)} - e^{-1})|\ge c t^{\alpha-1},
\]
and (ii) is proved.

\end{proof}

The following direct corollary of Theorem \ref{sharpHY1} shows
that the convergence rates for $\Delta^\varphi_{t,n}(A)$ and
$E^\varphi_{t,n}(A)$ obtained in the previous sections, e.g. the ones
in Theorems \ref{ThC}, \ref{Ex22}, \ref{IntAna} and \ref{Ex2}, are
optimal if the spectrum of $A$ is large enough.
\begin{cor}\label{corranall}
Let  $-A$ be the generator of a bounded $C_0$-semigroup
$(e^{-tA})_{t\ge 0}$ on a Banach space $X$ such that
$\overline{{\rm ran}}\, (A)=X$, and let   $\varphi \in \Phi,
\varphi(z)\not \equiv z.$
\begin{itemize}
\item [(i)]
If $ \{|s|:\,s\in \mathbb R,\,is\in \sigma(A)\}=\mathbb R_{+},$
then  there exist $c>0$ and $T>0$ such that for every $\alpha\in
(0,2]$ and  all $t \ge T,$
\begin{eqnarray*}
\|A^{-\alpha}\Delta^\varphi_{t,n} (A)\|&\ge& c \left(\frac{t}{n}\right)^{\alpha/2},\\
\|A^{-\alpha} E^\varphi_{t,n}(A)\|&\ge& c
\left(\frac{t^2}{n}\right)^{\alpha/2}.
\end{eqnarray*}
\item [(ii)] If
$ \mathbb R_{+}\subset \sigma(A), $ then  there exist $T>0$ and
$c>0$ such that for every $\alpha\in [0,2]$ and all $t\ge T,$
\begin{eqnarray*}
\|A^{-\alpha}\Delta^\varphi_{t,n}(A)\|&\ge& c n^{-1} t^{\alpha-1},\\
\|A^{-\alpha} E^\varphi_{t,n}(A)\|&\ge& c n^{-1} t^{\alpha}.
\end{eqnarray*}
\end{itemize}
\end{cor}

\section{Appendix: Estimates of $c_\beta$}

In this appendix we give fine estimates for the constants
$c_\beta$ introduced in  \eqref{denote}.

Fix $\beta >0.$
Recall that
\[
c_{\beta}=\|u_{\beta,\tau}\|_{\Wip(\C_{+})}
=\|u_{\beta,1}\|_{\Wip(\C_{+})}, \qquad \beta, \tau >0,
\]
where $u_{\beta,\tau}=\left(\frac{z}{z+\tau}\right)^\beta.$

We start with the expansion
\begin{equation}\label{ExF1}
1-(1-t)^\beta=-\sum_{n=1}^\infty \frac{(-\beta)_n}{n!}t^n,\quad
|t|<1,
\end{equation}
where
\[
(-\beta)_n:=\prod_{k=0}^{n-1} (k-\beta),\qquad n\ge 1,\quad (-\beta)_0:=1.
\]
Then, since
\[
\frac{1}{(1+z)^n}=\frac{1}{(n-1)!}\int_0^\infty e^{-zs} e^{-s}
s^{n-1}\,d s,\qquad z\in \C_+, \,\, n\in \N,
\]
it follows  from (\ref{ExF1}) that
\begin{eqnarray*}
1-\left(\frac{z}{z+1}\right)^\beta
&=&-\sum_{n=1}^\infty
\frac{(-\beta)_n}{n!}\frac{1}{(z+1)^n}\\
&=&-\sum_{n=1}^\infty \frac{(-\beta)_n}{n! (n-1)!}\int_0^\infty
e^{-z s} e^{-s} s^{n-1}\,d s\\
& =&\int_0^\infty e^{-z s}
q_{\beta}(s)\,ds,
\end{eqnarray*}
where
\begin{eqnarray*}
q_{\beta}(s)=\beta e^{-s}\sum_{n=0}^\infty
\frac{(1-\beta)_n}{(n+1)!}\frac{s^n}{n!}= \beta
e^{-s}\left._1F_1\right.(1-\beta;2;s),
\end{eqnarray*}
and $\left._1F_1\right.$ is
 a confluent hypergeometric function. (See \cite[Chapter 4]{AnAsRo99} for a background on confluent hypergeometric functions.)
  Hence
\[
c_{\beta}=1+\int_0^\infty |q_\beta(s)|\,ds= 1+\beta \int_0^\infty
e^{-s}\left._1F_1\right.(1-\beta;2;s)\,ds.
\]

Let now $\beta=m, m \in \N$. Then
\[
q_m(s)=m e^{-s}\left._1F_1\right.(1-m;2;s)=e^{-s}L_{m-1}^{(1)}(s),
\]
where $L_k^{(1)}(\cdot)$ are the (generalized) Laguerre
polynomials, and
\begin{equation}\label{cm}
c_m= 1+\int_0^\infty e^{-s}|L_{m-1}^{(1)}(s)|\,ds,\quad m\in \N.
\end{equation}
Thus estimates of $c_\beta,$ at least for integer $\beta,$
reduce to estimates of means of the Laguerre polynomials.
One of such estimates is provided by the following result.
\begin{theorem}\label{Lag1}
For every $m \in \N,$
\begin{equation}\label{LagEs}
\int_0^\infty e^{-s}|L_{m-1}^{(1)}(s)|\,ds\le 2\sqrt{m}.
\end{equation}
\end{theorem}

\begin{proof}
Recalling that
\[
\int_0^\infty e^{-s}s|L_{m-1}^{(1)}(s)|^2\,ds=m
\]
by \cite[Section 5.1]{Szego}, we have
\begin{align}\label{MGest}
\int_0^\infty e^{-s}|L_{m-1}^{(1)}(s)|\,ds \le& \int_0^1
e^{-s}|L_{m-1}^{(1)}(s)|(1-s)\,ds +\int_0^\infty
e^{-s}s|L_{m-1}^{(1)}(s)|\,ds\\
\le& \left(\int_0^1 e^{-s}(1-s)^2\,ds\right)^{1/2} \left(\int_0^1
e^{-s}|L_{m-1}^{(1)}(s)|^2\,ds\right)^{1/2}\notag \\
+&\left(\int_0^\infty s e^{-s}\,ds\right)^{1/2}\left(\int_0^\infty
e^{-s}s|L_{m-1}^{(1)}(s)|^2\,ds\right)^{1/2}\notag \\
\le& \sqrt{1-2e^{-1}} \left(\int_0^1
e^{-s}|L_{m-1}^{(1)}(s)|^2\,ds\right)^{1/2} +\sqrt{m}.\notag
\end{align}

By  Watson's identity \cite[p. 21]{Watson}, if   $s>0$ then
\begin{eqnarray*}
\frac{\pi}{2 m}e^{-s}|L_{m-1}^{(1)}(s)|^2
&=&\int_0^\pi L_{m-1}^{(1)}(2s(1+\cos\theta)) e^{-s(1+\cos\theta)}
\frac{\sin(s\sin\theta)}{s\sin\theta}\sin^2\theta \,d\theta\\
&=&\int_0^\pi L_{m-1}^{(1)}\left(4s\cos^2\frac{\theta}{2}\right) e^{-2s\cos^2\frac{\theta}{2}}
\frac{\sin(s\sin\theta)}{s}\sin\theta \,d\theta,
\end{eqnarray*}
hence
\[
\frac{\pi}{2 m}\int_0^1 |L_{m-1}^{(1)}(s)|^2 e^{-s}ds=\int_0^\pi
\sin\theta \, G_{m-1}(\theta)\,d\theta,
\]
where
\[
G_{m-1}(\theta):=\int_0^1 L_{m-1}^{(1)}\left(4s\cos^2\frac{\theta}{2}\right)
e^{-2s\cos^2\frac{\theta}{2}} \frac{\sin(s\sin\theta)}{s} \,ds, \qquad \theta \in [0,\pi].
\]
Using the properties
\[
L_{m-1}^{(1)}(s)=-\frac{d}{ds}L_m^{(0)}(s),\quad L_m^{(0)}(0)=1,
\]
and integrating by parts we obtain
\begin{eqnarray*}
&&4 \cos^2\frac{\theta}{2}\,G_{m-1}(\theta)= -\int_0^1
e^{-2s\cos^2\frac{\theta}{2}} \frac{\sin(s\sin\theta)}{s} \,d\left(
L_m^{(0)}\left(4s\cos^2 \frac{\theta}{2}\right)\right)\\
&=&- e^{-2\cos^2\frac{\theta}{2}} \sin(\sin\theta) L_m^{(0)}\left(4\cos^2
\frac{\theta}{2}\right) +\sin \theta\\
&-&2\cos^2\frac{\theta}{2} \int_0^1  e^{-2s\cos^2\frac{\theta}{2}}
\frac{\sin(s\sin\theta)}{s} L_m^{(0)}(4s\cos^2 \frac{\theta}{2}) \,ds\\
&+&\int_0^1  e^{-2s\cos^2\frac{\theta}{2}} \left[\sin\theta
\frac{\cos(s\sin\theta)}{s}- \frac{\sin(s\sin\theta)}{s^2}\right]
\, L_m^{(0)}\left(4s\cos^2 \frac{\theta}{2}\right)\,ds.
\end{eqnarray*}
Since  by  \cite[10.18(14)]{Erd2},
\[
|L_m^{(0)}(t)|\le  e^{t/2},\quad t\ge 0,
\]
and
\[
-\left(\frac{\sin\tau}{\tau}\right)'=\frac{\sin
\tau}{\tau^2}-\frac{\cos\tau}{\tau}>0,\qquad \tau\in (0,\pi/2),
\]
we  have
\begin{eqnarray*}
&& 4 \cos^2\frac{\theta}{2}\,G_{m-1}(\theta) \le \sin(\sin
\theta)+\sin\theta+2\cos^2\frac{\theta}{2}\int_0^1
\frac{\sin(s\sin\theta)}{s} \,ds\\
&+&\sin^2\theta \int_0^1
\left[\frac{\sin(s\sin\theta)}{s^2\sin^2\theta}-
\frac{\cos(s\sin\theta)}{s\sin\theta} \right]\,ds\\
&\le& \sin(\sin\theta)+\sin\theta +2\sin\theta \cos^2\frac{\theta}{2}
+\sin\theta\int_0^{\sin\theta}  \left[\frac{\sin\tau}{\tau^2}-
\frac{\cos\tau}{\tau}\right]\,d\tau\\
&=&\sin(\sin\theta)+\sin\theta+2 \sin\theta
\cos^2\frac{\theta}{2}+\sin\theta
\left[1-\frac{\sin(\sin\theta)}{\sin\theta}\right]\\
&=&2 \sin \theta \left(1+\cos^2\frac{\theta}{2}\right),\qquad \theta\in (0,\pi),
\end{eqnarray*}
hence
\[
|G_{m-1}(\theta)|\le \frac{\sin\frac{\theta}{2}(1+\cos^2\frac{\theta}{2})}{\cos\frac{\theta}{2}}.
\]

Then
\begin{eqnarray*}
\frac{\pi}{2 m}\int_0^1 |L_{m-1}^{(1)}(s)|^2 e^{-s}\,ds&\le&
\int_0^\pi \sin\theta\,|G_{m-1}(\theta)|\,d\theta\\
&\le& \int_0^\pi \sin\theta
\frac{\sin\frac{\theta}{2}(1+\cos^2\frac{\theta}{2})}{\cos\frac{\theta}{2}} \, d\theta\\
&=&2\int_0^\pi \sin^2\frac{\theta}{2}(1+\cos^2\frac{\theta}{2})\,d\theta\\
&=& 4\int_0^{\pi/2} (1+\cos^2\theta)\sin^2\theta\, d\theta\\
&=&2\left(B(1/2,3/2)+B(3/2,3/2)\right)\\
&=& \frac{5\pi}{4},
\end{eqnarray*}
where $B$ is the beta function, and
\[
\int_0^1 |L_{m-1}^{(1)}(s)|^2 e^{-s}\,ds \le \frac{5}{2}m.
\]
Therefore,  by (\ref{MGest}) we obtain that
\[
\int_0^\infty e^{-s}|L_{m-1}^{(1)}(s)|\,ds \le
\sqrt{m}+\sqrt{\left(1-\frac{2}{e}\right)\frac{5}{2}}\sqrt{m}\le
2\sqrt{m}.
\]

\end{proof}

\begin{cor}\label{Corc_m}
For all $\beta \ge 1$ and $m \in \mathbb N,$
\begin{equation}\label{cem}
c_m\le 1+2\sqrt{m}, \qquad c_\beta\le 2(1+2\sqrt{\beta}).
\end{equation}
\end{cor}
\begin{proof}
The bound for  $c_m, m \in \mathbb N,$ is a direct consequence of
Theorem \ref{Lag1}.

To estimate $c_\beta$ for $\beta \ge 1,$ we argue as in the proof of Lemma \ref{betaW}. If $\beta=m+\alpha, m \in \N, \alpha \in (0,1),$  then by \eqref{betebeta},
$$
c_\beta \le c_m c_\alpha \le 2(1+\sqrt m)\le 2(1+\sqrt \beta).
$$
\end{proof}

\end{document}